\documentclass[12pt,a4paper]{article}

\usepackage{amssymb}
\usepackage{amsthm}
\usepackage{amsfonts}
\usepackage{amsmath}
\usepackage{dsfont}
\usepackage{anysize}
\usepackage{hyperref}
\usepackage{setspace}
\usepackage{color}
\usepackage[top=1in, bottom=1in, left=1in, right=1in]{geometry}
\usepackage[style=alphabetic,maxnames=99,maxalphanames=5, isbn=false, giveninits=true, doi=false, url=false]{biblatex}

 \AtBeginBibliography{\footnotesize}
 
\renewbibmacro{in:}{}

\DeclareFieldFormat[article]{citetitle}{#1}
\DeclareFieldFormat[article]{title}{#1}
\DeclareFieldFormat[inbook]{citetitle}{#1}
\DeclareFieldFormat[inbook]{title}{#1}
\DeclareFieldFormat[incollection]{citetitle}{#1}
\DeclareFieldFormat[incollection]{title}{#1}
\DeclareFieldFormat[inproceedings]{citetitle}{#1}
\DeclareFieldFormat[inproceedings]{title}{#1}
\DeclareFieldFormat[thesis]{citetitle}{#1}
\DeclareFieldFormat[thesis]{title}{#1}
\DeclareFieldFormat[misc]{citetitle}{#1}
\DeclareFieldFormat[misc]{title}{#1}
\DeclareFieldFormat[book]{citetitle}{#1}
\DeclareFieldFormat[book]{title}{#1} 

\newtheorem*{theorem*}{Theorem}

\newtheorem{theorem}{Theorem}[section]

\newtheorem{lemma}[theorem]{Lemma}
\newtheorem{corollary}[theorem]{Corollary}

\newtheorem{remark}[theorem]{Remark}

\newtheorem{proposition}[theorem]{Proposition}

\usepackage[hyphenbreaks]{breakurl}
\newcommand{\PSH}{\mathrm{PSH}}
\usepackage{mathtools}

\DeclarePairedDelimiter\floor{\lfloor}{\rfloor}
\DeclareMathOperator{\vol}{vol}
\DeclareMathOperator{\usc}{usc}

\newcommand{\ddc}{\mathrm{dd}^{\mathrm{c}}}
\newcommand{\Nm}{\mathrm{Norm}}

\newcommand{\BM}{\mathrm{BM}}
\def\beq{\begin{equation}}
\def\eeq{\end{equation}}

\usepackage{changepage}

\bibliography{ComplexGeometry}

\title{ The volume of pseudoeffective line bundles and partial equilibrium}
\author{Tam\'as Darvas (University of Maryland) \\ Mingchen Xia (Sorbonne Universit\'e)}

\date{}
\begin{document}
\maketitle

\begin{abstract}
Let $(L,he^{-u})$ be a pseudoeffective line bundle on an $n$-dimensional compact Kähler manifold $X$. Let $h^0(X,L^k\otimes \mathcal I(ku))$ be the dimension of the space of sections $s$ of $L^k$ such that $h^k(s,s)e^{-ku}$  is integrable. We show that the limit of $k^{-n}h^0(X,L^k\otimes \mathcal I(ku))$ exists, and equals the non-pluripolar volume of $P[u]_\mathcal I$, the $\mathcal I$-model potential associated to $u$. We give applications of this result to Kähler quantization: fixing a Bernstein--Markov measure $\nu$, we show that the partial Bergman measures of $u$ converge weakly to the non-pluripolar Monge--Ampère measure of $P[u]_\mathcal I$, the partial equilibrium.   
\end{abstract}

\section{Introduction}

A major theme in Kähler geometry has been  the approximation/quantization
of natural objects in the theory, going back to a problem of Yau \cite{Ya86} and early work of Tian  \cite{Ti88}.   
 Initial focus was on the quantization of smooth Kähler metrics, with asymptotic expansion results due to Tian, Bouche, Catlin, Zelditch, Lu and others \cite{Ti90,Bo90,Ca97,Ze98,Lu00}. Later, Donaldson proposed to not just quantize Kähler metrics, but their infinite dimensional geometry as well \cite{Don01}. This led to a flurry of activity helping to better understand notions of stability in K\"ahler geometry (see \cite{Bern18}, \cite{CS12}, \cite{PS06}, \cite{SZ10}, \cite{DLR20}, \cite{Zh21} to only mention a few works in a fast expanding literature). We refer to the excellent textbook \cite{MM07} for a detailed discussion and history of many classical results in this direction.

Our work fits into this broad context, however we consider perhaps the most singular objects one can work with: positively-curved metrics on a pseudoeffective line bundle. Despite the fact that potentials of these positively-curved metrics are only integrable in general, we will be able to recover their  volumes and partial equilibrium measure using quantization, significantly extending the scope of previous results in the literature. 

\paragraph{The volume of a pseudoeffective line bundle.} We now describe our results. Let $L$ be a holomorphic line bundle on a compact connected Kähler manifold $(X,\omega)$ of dimension $n$. Let $h$ be a smooth metric on $L$, and let $\theta:=c_1(L,h)$ denote the Chern form of $h$. Let $(T,h_T)$  be an arbitrary Hermitian holomorphic vector bundle on $X$ of rank $r$, that will be used to \emph{twist} powers of $L$. 

By $\PSH(X,\theta)$ we will denote the space of quasi-plurisubharmonic (quasi-psh) functions $v$ on $X$ such that $\theta + \ddc v = \theta + \frac{\mathrm{i}}{2\pi} \partial \bar \partial v \geq 0$ in the sense of currents. Here $\textup{d} = \partial + \bar \partial$ and $ \textup{d}^c = \frac{\mathrm{i}}{4\pi}(-\partial + \bar \partial)$.

A priori $\PSH(X,\theta)$ may be empty, but if there exists $u \in \PSH(X,\theta)$, then following terminology of Demailly, we say that the pair $(L,h e^{-u})$ is a pseudoeffective (psef) Hermitian line bundle. Moreover, to such $u$ one can associate a non-pluripolar complex Monge--Ampère measure $\theta_u^n$, as introduced in \cite{BEGZ10},\cite{GZ07}, following ideas by Cegrell \cite{Ce98} and Bedford--Taylor \cite{BT76} in the local case (see Section~\ref{subsec:singtype} for more details). 

We can associate to $u$ the so called \emph{$\mathcal I$-model potential/envelope} $P[u]_\mathcal I \in \PSH(X,\theta)$:
\begin{equation}\label{eq: PIdef}
P[u]_\mathcal I := \sup \left\{w \in \PSH(X,\theta), \ w \leq 0, \ \mathcal I(tw) \subseteq \mathcal I (tu), \ t \geq 0\right\}\,.
\end{equation}
Here $\mathcal I(tu)$ is a multiplier ideal sheaf, locally generated by holomorphic functions $f$ such that $|f|^2e^{-tu}$ is integrable.  To our knowledge $P[u]_\mathcal I$ was first considered in \cite{KS20}, and we studied it in detail in \cite[Section~2.4]{DX22} and also \cite{Tr20}.

Let $H^0(X,T \otimes L^k \otimes \mathcal I(ku))$ be the space of global holomorphic sections $s$ of $T \otimes L^k$ satisfying $\int_X h_T \otimes h^k(s,s)e^{-ku} \omega^n <\infty$. We also introduce the notation 
\[
h^0(X,T \otimes L^k \otimes \mathcal I(ku)) := \dim_{\mathbb C} H^0(X,T \otimes L^k \otimes \mathcal I(ku))\,.
\]

It was conjectured by Cao and Tsuji that $\lim_{k \to \infty} \frac{1}{k^n}h^0(X,T \otimes L^k \otimes \mathcal I(ku))$ always exists (\cite[page 7]{Cao14}, \cite[Section 4.4]{Ts07}). We show that this is indeed the case, and we give a precise formula for the limit in terms of the non-pluripolar volume of $P[u]_\mathcal I$:

\begin{theorem}\label{thm:vol_formula_main} Let  $(L,h e^{-u})$ be a pseudoeffective Hermitian line bundle on $X$, and let $T$ be a holomorphic vector bundle of rank $r$ on $X$. Then 
\begin{equation}\label{eq: main volume_eq}
\lim_{k \to \infty} \frac{1}{k^n}h^0(X,T \otimes L^k \otimes \mathcal I(ku)) = \frac{r}{n!} \int_X \theta_{P[u]_{\mathcal I}}^n\,.
\end{equation}
\end{theorem}

When $L$ is ample and $T$ is a line bundle, Theorem~\ref{thm:conv_Bergman_equi_main} was obtained using non-Archimedean methods in \cite[Theorem~1.4]{DX22}. As these techniques do not extend to the pseudoeffective case, we take a more elementary approach in this work. In addition, in Section~\ref{subsec:Rlinebundle} we show that the analogue of Theorem~\ref{thm:vol_formula_main} holds for pseudoeffective $\mathbb R$-line bundles as well.

In case $u$ has analytic singularity type with smooth remainder (see Section~\ref{subsec:metricsingtype} for the definition), formula \eqref{eq: main volume_eq} is a well-known consequence of the Riemann--Roch theorem of Bonavero \cite[Théorème~1.1, Corollaire~1.2]{Bon98} (see \cite[Theorem~2.26]{DX22}). In this case, it is possible to apply a resolution of singularities to simplify/principalize the singularity locus of  $\mathcal I(u)$, allowing for a precise asymptotic analysis. In addition, in this case one also has $\int_X \theta_{P[u]_{\mathcal{I}}}^n=\int_X \theta_{u}^n$ \cite[Proposition 2.20]{DX22}, simplifying the right-hand side of \eqref{eq: main volume_eq}. However, for general $u \in \textup{PSH}(X,\theta)$, one is forced to use the measures $\theta_{P[u]_{\mathcal{I}}}^n$, and this is one of the novelties of our work. Indeed, since $u - \sup_X u \leq P[u]_\mathcal I$,  \cite[Theorem~1.1]{WN19} gives that $ \int_X \theta_{u}^n \leq  \int_X \theta_{P[u]_{\mathcal I}}^n$, and strict inequality is possible, as pointed out in \cite[Example~2.19]{DX22}.

Formula \eqref{eq: main volume_eq} is also known for $u := V_\theta:=\sup\left\{\varphi\in\PSH(X,\theta):\varphi\leq 0 \right\}$, the potential with minimal singularity type in $\PSH(X,\theta)$ \cite[Proposition~1.18]{BEGZ10}. In this case we again have $\int_X \theta_{P[V_\theta]_{\mathcal{I}}}^n=\int_X \theta_{V_\theta}^n$, recovering Boucksom's formula \cite{Bo02,BEGZ10}:
\[
\lim_{k \to \infty} \frac{1}{k^n}h^0(X, L^k) = \frac{1}{n!} \int_X \theta_{V_\theta}^n\,.
\]
The above expression is called the volume of the line bundle $L$ in the literature \cite{Bou02,Dem12}, justifying our terminology to call $\frac{1}{n!}\int_X \theta_{P[u]_{\mathcal I}}^n$  the \emph{volume} of the pair $(L,he^{-u})$.

As $T$ is allowed to be an arbitrary vector bundle in Theorem \ref{thm:vol_formula_main}, one can hypothesize a version of this result with $T$ being a coherent sheaf on $X$. This was pointed out to us by L\'aszl\'o Lempert.

At slight expense of precision, we briefly describe the strategy behind the proof of Theorem~\ref{thm:vol_formula_main}. By \cite[Theorem~1.2]{WN19}, both the left and right sides of \eqref{eq: main volume_eq} only depend on the singularity type of the potential $u$. As a result, we can use the metric topology of singularity types introduced in \cite{DDNL5}, and further developed in \cite{DX22}. Let us very briefly recall the terminology. For $v,w \in \PSH(X,\theta)$  we say that 
\begin{itemize}\vspace{-0.1cm}
	\item $v$ is more singular than $w$, ($v \preceq w$) if there exists $C\in \mathbb{R}$  such that $v\leq w+C$;\vspace{-0.1cm}
	\item $v$ has the same singularity type as $w$, ($v \simeq w$), if $v\preceq w$ and $w\preceq v$. \vspace{-0.1cm}
\end{itemize}
The classes $[v] \in \mathcal S := \PSH(X,\theta)/\simeq$ of this latter equivalence relation are called \emph{singularity types}. As pointed out in \cite{DDNL5}, and recalled in Section~\ref{subsec:metricsingtype}, $\mathcal S$ admits a natural pseudometric $d_\mathcal S$, making $(\mathcal S,d_\mathcal S)$ complete (in the presence of positive mass).

By \cite[Proposition~2.20]{DX22}, we have $H^0(X,T \otimes L^k \otimes \mathcal I(ku))=H^0(X,T \otimes L^k \otimes \mathcal I(kP[u]_{\mathcal I}))$ and  $P[u]_{\mathcal I}=P[P[u]_{\mathcal I}]_\mathcal I$ (i.e., $u \to P[u]_{\mathcal I}$ is a projection). Hence, it is enough to prove \eqref{eq: main volume_eq} for potentials of the form $P[u]_{\mathcal I}$. In Section~\ref{sec:closanasing} we show that the singularity types  $[P[u]_{\mathcal I}] \in \mathcal S$ can be $d_\mathcal S$-approximated by analytic singularity types $[u_j] \in \mathcal S$. It is crucial to work with potentials of the form $P[u]_{\mathcal I}$, as the same property does not hold for general potentials $u$.

The proof is then completed by an approximation argument. We take a decreasing sequence $u_j \in \PSH(X,\theta)$ composed of potentials with analytic singularity types such that $d_\mathcal S([u_j],[u]) \to 0$. By Bonavero's theorem we know that \eqref{eq: main volume_eq} holds for each $u_j$. It is known that $d_\mathcal S([u_j],[u]) \to 0$ implies $\int_X \theta_{P[u_j]_\mathcal I}^n \to \int_X \theta_{P[u]_{\mathcal I}}^n$, and we will prove a similar convergence result for the left-hand side of \eqref{eq: main volume_eq} as well, to finish the argument.

Let us mention applications of Theorem~\ref{thm:vol_formula_main} that are treated elsewhere. By \cite{LM09, KK12} we can naturally assign a family of convex Okounkov bodies $\Delta(L)$ to a given big line bundle $L$, depending only on the numerical class of $L$. Moreover, $\vol L=\vol \Delta(L)$. In \cite{Xia21b}, based on Theorem~\ref{thm:vol_formula_main}, the second-named author extended this construction to Hermitian pseudoeffective line bundles: it is possible to define a natural family of convex bodies $\Delta(L,\phi)$ associated with a given Hermitian pseudoeffective line bundle $(L,\phi)$ such that $\vol \Delta(L,\phi)=\vol (L,\phi)$. 

Another application concerns automorphic forms. Consider an automorphic line bundle $L$ on a Shimura variety or mixed Shimura variety $X$. The global sections of $L^k$ correspond to certain automorphic forms. It is a natural and important question in number theory to understand the asymptotic dimensions of these automorphic forms. In general, $X$ is not compact, but it admits natural smooth compactifications \cite{AMRT}. Usually the smooth equivariant metrics on $L$ only extends to singular metrics on a compactification. In this case, Theorem~\ref{thm:vol_formula_main} can be naturally applied. In the special case of Siegel--Jacobi modular forms, this idea has been carried out concretely in the recent preprints \cite{BBGHdJ, BBGHdJ2}. Using a particular case of  Theorem~\ref{thm:vol_formula_main}, Botero--Burgos Gil--Holmes--de Jong managed to prove that the ring of Siegel--Jacobi modular forms is not finitely generated, disproving a well known claim by Runge \cite{Run95}.

\paragraph{Convergence of partial Bergman measures.} As another application of Theorem~\ref{thm:vol_formula_main}, we give a very general convergence result for partial Bergman measures to the partial equilibrium, extending the scope of numerous results in the literature. 

First we recall terminology introduced in \cite{BB10}. A \emph{weighted subset} of $X$ is a pair $(K,v)$ consisting of a closed non-pluripolar subset $K\subseteq X$ and a function $v\in C^0(K)$.
Next, given $u \in \PSH(X,\theta)$, we tailor the definition of $\mathcal I$-model envelope from \eqref{eq: PIdef} to the pair $(K,v)$: 
\begin{equation}\label{eq: PE_potential}
P[u]_{\mathcal{I}}(v):= \usc \big(\sup\left\{ w \in \PSH(X,\theta): w|_K \leq v  \textup{ and }   \mathcal I(tw) \subseteq \mathcal I (tu), \ t \geq 0\right\}\big)\,.
\end{equation}
Here $\usc(\cdot )$ denotes the least upper semi-continuous envelope. In case $K=X$, $\usc(\cdot )$ is unnecessary, moreover we have $P[u]_{\mathcal{I}}(0)=P[u]_{\mathcal{I}}$.

As a consequence of Corollary~\ref{cor:suppthetan} below,  $\theta_{P[u]_{\mathcal{I}}(v)}^n$ does not put mass on the set $(X \setminus K) \cup \{P[u]_{\mathcal{I}}(v) < v\}$. What is more, when $K=X$ and $v \in C^2(X)$, the main result of \cite{DNT19} implies that 
\[
\theta_{P[u]_{\mathcal{I}}(v)}^n = \mathds{1}_{\{P[u]_{\mathcal{I}}(v) = v\}} \theta_v^n\,.
\]
Analogous properties of equilibrium type measures in different contexts were obtained in \cite{SZ03, Brm11,RWN17}. With this in mind, we will call the measure $\theta_{P[u]_{\mathcal{I}}(v)}^n$ the \emph{partial equilibrium (measure)} associated to $u$ and $(K,v)$. Theorem~\ref{thm:conv_Bergman_equi_main} will further justify  this choice of terminology.

Let $(T,h_T)$ be a Hermitian line bundle. Let $\nu$ be a Borel probability measure on $K$.
We consider the following norms on $H^0(X,L^k \otimes T)$:
\[
\begin{aligned}
N^k_{v,\nu}(s) :=& \left(\int_K h^k \otimes h_T(s,s) e^{-kv} \,\mathrm{d}\nu\right)^{\frac{1}{2}}\,, \ \ \ \ N^k_{v,K}(s) :=& \sup_{K} \big( h^k \otimes h_T(s,s)e^{-kv}\big)^{\frac{1}{2}}\,.
\end{aligned}
\]
Note that  we always have  $N^k_{v,\nu}(s)\leq  N^k_{v,K}(s)$.
The measure  $\nu$  is a \emph{Bernstein--Markov measure} with respect to $(K,v)$ if  for each $\varepsilon>0$, there is a constant $C_\varepsilon>0$ such that
\begin{equation}\label{eq:BM_main}
     N^k_{v,K}(s) \leq C_\varepsilon e^{\varepsilon k} N^k_{v,\nu}(s)
\end{equation}
for any $s \in H^0(X,L^k \otimes T)$. A broad class of Bernstein--Markov measures are probability volume forms with respect to $(X,v)$, where $v \in C^\infty(X)$. For more complicated examples we refer to \cite[Section 1.2]{BBWN11}.

We introduce the associated \emph{partial Bergman kernels}: for any $k\in \mathbb{N}$, $x\in K$,
 \begin{equation}\nonumber
     B^k_{v,u, \nu}(x) := \sup \left\{h^k \otimes h_T(s,s)e^{-kv}(x): N^k_{v,\nu}(s,s) \leq 1\,,  s \in H^0(X,L^k \otimes T \otimes \mathcal I(ku)) \right\}\,.
 \end{equation}
The associated partial Bergman measures on $X$ are identically zero on $X \setminus K$ and on $K$ are defined in the following manner 
\begin{equation}
\beta^k_{v,u,\nu} : = \frac{n!}{k^n} B^k_{v,u,\nu} \,\mathrm{d}\nu\,.
\end{equation}

Our next result states that the partial Bergman measures $\beta^k_{v,u,\nu}$ quantize the 
non-pluripolar measure $\theta_{P[u]_{\mathcal{I}}(v)}^n$, the partial equilibrium of this setting:
\begin{theorem}\label{thm:conv_Bergman_equi_main} Let  $(L,h e^{-u})$ be a pseudoeffective Hermitian holomorphic line bundle on $X$, and let $(T,h_T)$ be a Hermitian line bundle. Suppose that $\nu$ is a Bernstein--Markov measure with respect to a weighted subset $(K,v)$. Then $\beta^k_{v,u,\nu} \rightharpoonup \theta_{P[u]_{\mathcal{I}}(v)}^n$ weakly, as $k \to \infty$.
\end{theorem}

To our knowledge, this result is new even in the case when $L$ is assumed to be ample. 
An important particular case is when $T$ is trivial, $K=X$, $v\equiv 0$ and $\mu = \omega^n/\int_X \omega^n$. In this case we simply denote $\beta^k_{u}:=\beta^k_{0,u,\omega^n}$ and recall that $P[u]_{\mathcal{I}} = P[u]_{\mathcal{I}}(0)$. We have the following corollary:
\begin{corollary}\label{cor: main}
For $u \in \PSH(X,\theta)$ we have that $\beta^k_{u} \rightharpoonup \theta_{P[u]_{\mathcal{I}}}^n$ weakly, as $k \to \infty$.
\end{corollary}

When $T$ is the trivial line bundle and  $u$ has minimal singularity, Theorem~\ref{thm:conv_Bergman_equi_main} recovers \cite[Theorem~B]{BBWN11}. As part of our argument, in Section~\ref{sec:env} and Section~\ref{sec:quant} we also extend \cite[Theorems~A and B]{BB10} to our partial setting. We suspect that using our results one can now prove equidistribution theorems for (partial) Fekete point configurations, extending \cite[Theorem A]{BBWN11} to our context. However, to stay brief we omit this discussion here.

When $T$ is the trivial line bundle, $K=X$, $v \in C^2(X)$, $\mu = \omega^n/\int_X \omega^n$ and $u$ has  minimal or exponentially continuous  singularity type, we are essentially in the setting of \cite[Theorem~1.4]{Brm11} and \cite[Theorem~1.4]{RWN17}.
Our Theorem~\ref{thm:conv_Bergman_equi_main} extends these results, to the extent that our singular setting allows. Indeed, as $[u]$ is of $\mathcal I$-model type in these cases, we automatically get that $P[u](v) = P[u]_\mathcal I(v)$, where 
\[
P[u](v) := \usc \sup\left\{ h \in \PSH(X,\theta): \ h \leq v, \  [h] \preceq [u]\right\}\,.
\]
See Section~\ref{sec:closanasing} for more details. Hence, in this case the (partial) Bergman measures converge weakly to $\theta_{P[u](v)}^n$. In \cite{Brm11,RWN17} the authors actually argue pointwise convergence of the density functions as well, on the locus where $P[u](v) = v$ and $\theta_v>0$. As our $v$ in Theorem~\ref{thm:conv_Bergman_equi_main} is only continuous, it is not clear how to interpret the condition $\theta_v>0$ in our context.

Observe that $\int_X \beta^k_{v,u,\nu}=n! k^{-n}h^0(X,L^k \otimes T \otimes \mathcal I(ku))$. In particular, Theorem~\ref{thm:conv_Bergman_equi_main} recovers Theorem~\ref{thm:vol_formula_main} after an integration. In fact, this plays a crucial role in the argument of Theorem~\ref{thm:conv_Bergman_equi_main}. As all the measures $\beta^k_{v,u,\mu}$ have uniformly bounded masses, they form a weakly compact family. The difficulty is to prove that each  subsequential limit measure is dominated by $\theta_{P[u]_{\mathcal{I}}(v)}^n$. Then the argument is concluded by simply comparing total masses of the limit measures.

The literature on partial Bergman kernels/measures has been fast expanding in many directions. One particular line of study concerns partial Bergman kernels arising from sections vanishing along a smooth divisor $V$, with the vanishing order increasing in the large limit. As pointed out in numerous works mentioned below, this setup is closely related to ours, when one considers $L^2$ integrable sections with respect to a weight that has logarithmic singularity along $V$. It would be interesting to study this connection in the future. One of the first works on this topic was that of Berman \cite{Brm11}, who proved $L^1$ convergence of the volume densities of the partial Bergman measures. Ross--Singer \cite{RS17} and Zelditch--Zhou \cite{ZeZh19} considered this problem in the presence of an $S^1$-symmetry near the vanishing locus, identified the forbidden region in terms of the Hamiltonian action, and gave detailed asymptotic expansions. When symmetries are not present, Coman--Marinescu \cite{CM17} proved that the partial Bergman kernel has exponential decay near the vanishing locus.
For recent extensions to smooth and singular subvarieties $V$, see \cite{CMN19, Su20}.

Applications of partial Bergman kernels related to  test configurations and geodesic rays were explored in \cite{RWN14} and \cite{DX22}.

In another line of study, Zelditch--Zhou initiated the study of partial Bergman kernels that arise from spectral subspaces of the Toeplitz quantization of a smooth Hamiltonian \cite{ZeZh20}. They showed that their partial density of states also converges to an equilibrium type measure, suggesting possible connections with our Theorem~\ref{thm:conv_Bergman_equi_main}. Specifically, given the Hamiltonian data $(H,E)$ of \cite{ZeZh20}, we wonder if there exists $v \in C^\infty(X)$ and $u \in \textup{PSH}(X,\theta)$ such that $\{H(z) < E\} = \{P[u]_\mathcal I(v)=v\}$. If the answer to this question is affirmative, then using the terminology of \cite[Main Theorem]{ZeZh20} we would obtain that 
$\Pi_{k,\mathcal S_k} \omega^n \rightharpoonup \theta^n_{P[u]_\mathcal I(v)}$.

\paragraph{Acknowledgements. }
We would like to thank Bo Berndtsson, Junyan Cao, Jakob Hultgren, L\'aszl\'o Lempert, Yaxiong Liu, Duc-Viet Vu, and Steven Zelditch for discussions related to the topic of the paper.  We thank the anonymous referee for suggesting many improvements.
The first author was partially supported by an Alfred P. Sloan Fellowship and National Science Foundation grant DMS-1846942.

\paragraph{Organization.} In Section~\ref{sec:pre} we recall the relevant notions of envelopes, and adapt results in the literature about the metric topology of singularity types to our context. In Section~\ref{sec:closanasing} we characterize the closure of analytic singularity types in a big cohomology class. In Section~\ref{sec:proofmainthm} we prove Theorem~\ref{thm:vol_formula_main}. In Section~\ref{sec:env} and Section~\ref{sec:quant} we extend the related results of \cite{BB10} and \cite{BBWN11} to our partial context, and prove Theorem \ref{thm:conv_Bergman_equi_main}.

\section{Preliminaries}\label{sec:pre}

\subsection{Non-pluripolar products and singularity types}\label{subsec:singtype}
Let $X$ be a compact Kähler manifold. Let $\theta$ be a smooth real $(1,1)$-form on $X$. 
Let $\PSH(X,\theta)$ be the set of $\theta$-plurisubharmonic ($\theta$-psh) functions on $X$. Assume that the cohomology class of $\theta$ is pseudoeffective, i.e., that $\PSH(X,\theta)$ is non-empty.

Let $V_\theta := \sup\{v \in \PSH(X,\theta) : v \leq 0\}$ be the potential with minimal singularity in $\PSH(X,\theta)$. We recall the construction of non-pluripolar product associated to $u_1,\ldots, u_n \in \PSH(X,\theta)$ from \cite{BEGZ10}.

Let $k \in \mathbb  N$. Using Bedford--Taylor theory \cite{BT76}, one can consider the  following sequence of measures on $X$: 
\[
\mathds{1}_{\bigcap_j\{\varphi_j>V_{\theta}-k\}}(\theta+\ddc \max(\varphi_1, V_{\theta}-k))\wedge\cdots\wedge (\theta+\ddc\max(\varphi_n, V_{\theta}-k))\,.
\] 
It has been shown in \cite[Section~1]{BEGZ10} that these measures 
converge weakly to the so called \emph{non-pluripolar product}  $\theta_{\varphi_1 } \wedge\cdots\wedge\theta_{\varphi_n }$, as $k \to \infty$. All complex Monge--Ampère measures will be interpreted in this sense in our work.

The resulting positive measure $\theta_{\varphi_1 } \wedge\cdots\wedge\theta_{\varphi_n }$ does not charge pluripolar sets. The particular case when $u := u_1 = \cdots = u_n$ will yield $\theta_{u}^n$, the non-pluripolar complex Monge--Ampère measure of $u$.

For any $u \in \PSH(X,\theta)$, let $\mathcal{I}(u)$ denote Nadel's multiplier ideal sheaf of $u$, namely, the coherent ideal sheaf of holomorphic functions $f$, such that $|f|^2e^{-u}$ is integrable. These objects allow to introduce an algebraic refinement of the notion of singularity type from the introduction. For $u,v\in \PSH(X,\theta)$ we have the following relations:
    \begin{itemize}\vspace{-0.1cm}
	\item $u \preceq_\mathcal I v$ (also written as $[u] \preceq_\mathcal I [v]$) if  $\mathcal I(tu)\subseteq \mathcal I(tv)$ for all $t>0$;\vspace{-0.1cm}
	\item $ u \simeq_\mathcal I v$ 
	(also written as $[u] \simeq_\mathcal I [v]$) if $u\preceq_\mathcal I v$ and $v\preceq_\mathcal I u$. \vspace{-0.1cm}
\end{itemize}
The relation $\simeq_\mathcal I$ induces equivalence classes called $\mathcal I$\emph{-singularity types} $[u]_\mathcal I$, for any $u \in \PSH(X,\theta)$. As pointed out in \cite{DX22},  $[u]=[v]$ implies $[u]_\mathcal I=[v]_\mathcal I$, but not vice versa.  

The different equivalence relations ($\simeq$ and $\simeq_\mathcal I$) admit two different envelope notions, as already alluded to in the introduction. Let us revisit them in a very general setup, that will be needed later. Let $K \subseteq X$ compact and non-pluripolar, and let $v:K \to [-\infty,\infty]$ measurable.
To such $v$ and $u \in \PSH(X,\theta)$ we associate the following notion of envelope:
\[
\begin{aligned}
P_K^{\theta}[u](v)&:=\usc\big(\sup \left\{\, w\in \PSH(X,\theta): [w]\preceq [u], w|_K \leq v \,\right\} \big)\,,\\
P_K^{\theta}[u]_{\mathcal{I}}(v)&:=\usc\big(\sup \left\{\, w\in \PSH(X,\theta): [w]\preceq_{\mathcal{I}} [u], w|_K \leq v \,\right\} \big)\,.
\end{aligned}
\]
Here and later $\usc(\cdot)$ denotes the upper semi-continuous regularization.
We omit $\theta$ and $X$ from our notations when there is no risk of confusion. In addition, we will use the following shorthand notation, ubiquitous in the literature:
\[
P[u]:=P^{\theta}_X[u](0)\,,\quad P[u]_{\mathcal{I}}:=P^{\theta}_X[u]_{\mathcal{I}}(0)\,.
\]
A potential $u\in \PSH(X,\theta)$ is \emph{model} if $u=P[u]$, and it is  \emph{$\mathcal{I}$-model} if $u=P[u]_{\mathcal{I}}$.

For any usc function $f:X \to [-\infty,\infty)$ we define
\begin{equation}
P^{\theta}(f):=\usc\sup\left\{\,\varphi\in \PSH(X,\theta):\varphi\leq f \,\right\}\,.
\end{equation}
Building on the above, for usc functions $f_1,\ldots,f_N$ we define a notion of rooftop envelope:
\[
P^{\theta}(f_1,\ldots,f_N):=P^{\theta}\left(\min\{f_1,\ldots,f_N\} \right)\,.
\]

The following lemma was essentially proved in \cite{DDNL5}. We recall the short proof as a courtesy to the reader:

\begin{lemma}\label{lem:rooftop_mod_I_mod} Let $u,v \in \PSH(X,\theta)$ such that $P^{\theta}(u,v) \in \PSH(X,\theta)$. If $u,v$ are model (resp. $\mathcal I$-model) then $P^{\theta}(u,v)$ is also model (resp. $\mathcal I$-model).
\end{lemma}

\begin{proof} Since $P(u,v) \leq \min(u,v)$, we get that $P[P(u,v)] \leq P[u] = u$ and 
$P[P(u,v)] \leq P[v] = v$, hence $P[P(u,v)] \leq P(u,v)$. This implies $P[P(u,v)] = P(u,v)$ as desired. The statement about $\mathcal I$-model potentials is proved the same way.
\end{proof}

For any $x\in X$ and $u \in \textup{PSH}(X,\theta)$ we  denote by $\nu(u,x)$ the Lelong number of $\varphi$ at $x$. 
We recall the following result from \cite{BFJ08}, adapted to our context in \cite[Corollary~2.16]{DX22}:
\begin{proposition}\label{prop:Ilelong}
Let $u,v\in \PSH(X,\theta)$. Then\\
\noindent (i) $[u]\preceq_{\mathcal{I}} [v]$ if and only if for any smooth modification $\pi:Y\rightarrow X$, any $y\in Y$, $\nu(\pi^*u,y)\geq \nu(\pi^*v,y)$.\\
\noindent (i) $[u]\simeq_{\mathcal{I}} [v]$ if and only if for any smooth modification $\pi:Y\rightarrow X$, any $y\in Y$, $\nu(\pi^*u,y)= \nu(\pi^*v,y)$.
\end{proposition}
\begin{corollary}\label{cor:linearIcomp}
Let $u_0,u_1,v_0,v_1\in \PSH(X,\theta)$ with $[u_0]\preceq_{\mathcal{I}} [v_0]$, $[u_1]\preceq_{\mathcal{I}} [v_1]$. For any $t\in [0,1]$ we have $[(1-t)u_0+tu_1]\preceq_{\mathcal{I}} [(1-t)v_0+tv_1]$.
\end{corollary}
\begin{proof}
This follows from Proposition~\ref{prop:Ilelong}(i) and the additivity of Lelong numbers \cite[Corollary 2.10]{Bo17}.
\end{proof}

Lastly, we show concavity properties for the envelopes defined above:

\begin{proposition}\label{prop:concmono}
Let $v\in C^0(K)$, $u_0,u_1\in \PSH(X,\theta)$. The following hold:\vspace{0.1cm}\\
\noindent (i) For any $t\in [0,1]$, let $u_t=tu_1+(1-t)u_0$, then
    \begin{flalign}\label{eq:conca}
    tP_{K}[u_1]_{\mathcal{I}}(v)+(1-t)P_{K}[u_0]_{\mathcal{I}}(v) &\leq P_{K}[u_t]_{\mathcal{I}}(v)\,,\\
    tP_{K}[u_1](v)+(1-t)P_{K}[u_0](v) &\leq P_{K}[u_t](v)\,.\nonumber
    \end{flalign}
    
\noindent (ii) Assume that $[u_0]\preceq_{\mathcal{I}} [u_1]$ (resp. $[u_0]\preceq [u_1]$), then  $P_{K}[u_0]_{\mathcal{I}}(v)\leq P_{K}[u_1]_{\mathcal{I}}(v)$ (resp. $P_{K}[u_0](v)\leq P_{K}[u_1](v)$).
\end{proposition}
\begin{proof}The proof of (ii) follows from the  definitions.
To prove (i),  let $h_0,h_1\in \PSH(X,\theta)$, such that $[h_i]\preceq_{\mathcal{I}}[u_i]$, $h_i|_K\leq v$. Then by Corollary~\ref{cor:linearIcomp}, $[th_1+(1-t)h_0]\preceq_{\mathcal{I}}[u_t]$. It is clear that $th_1|_K+(1-t)h_0|_K\leq v$. Hence, $th_1+(1-t)h_0\leq P_{K}[u_t]_{\mathcal{I}}(v)\,.$

As $h_1$ and $h_0$ are arbitrary candidates, we conclude the first inequality in \eqref{eq:conca}. The proof of the second inequality is similar.
\end{proof}

\subsection{The metric topology of singularity types}\label{subsec:metricsingtype}

Let $\mathcal S(X,\theta)$ be the set of singularity types of $\theta$-psh functions: $\mathcal S(X,\theta):= \PSH(X,\theta)/\simeq$. Let $\mathcal{A}(X,\theta)\subseteq \mathcal S(X,\theta)$ be the set of \emph{analytic singularity types}, namely, all singularity types $[u]$ represented by $u \in \PSH(X,\theta)$ such that $u$ is locally of the following form:
\begin{equation}\label{eq: analytic_def}
u = c\log \sum_{i=1}^N|f_i|^2+g
\end{equation}
where $c \in \mathbb  Q^+$, $f_1,\ldots,f_N$ are holomorphic functions, and  $g$ is a  bounded function. When $g$ can be taken to be smooth, then following \cite{Dem18a} we say that $[u]$ is a \emph{neat} analytic singularity type.

In \cite{DDNL5} the authors constructed a  pseudometric $d_{\mathcal{S}}$ on $\mathcal S(X,\theta)$. As we will use the $d_\mathcal S$ topology extensively in this work, we recall here a few basic facts, and refer to \cite{DDNL5} for a more complete picture. 

The definition of $d_\mathcal S$ involves embedding $\mathcal S(X,\theta)$ into the space of $L^1$ geodesic rays \cite[Section~3]{DDNL5}. We do not recall the exact definition, but simply recall that there is a constant $C>0$ depending only on $n$, such that for any $[u],[v]\in \mathcal S(X,\theta)$, we have
\begin{equation}
    d_{\mathcal{S}}([u],[v])\leq \sum_{j=0}^n \left(2\int_X \theta_{V_\theta}^j\wedge\theta_{\max\{u,v\}}^{n-j}-\int_X \theta_{V_\theta}^j\wedge\theta_{u}^{n-j}-\int_X \theta_{V_\theta}^j\wedge\theta_{v}^{n-j} \right)\leq C d_{\mathcal{S}}([u],[v])\,.
\end{equation}
Note that the term in middle is independent of the choices of representatives $u$ and $v$, as a consequence of \cite[Theorem~1.1]{DDNL2}.

\begin{theorem}[{\cite[Theorem~1.1]{DDNL5}}]
For any $\delta>0$, the space 
\[
\mathcal{S}_{\delta}(X,\theta):=\left\{\,[u]\in \mathcal{S}(X,\theta),\int_X \theta_{u}^n\geq \delta \,\right\}
\]
is $d_{\mathcal{S}}$-complete.
\end{theorem}

We paraphrase {\cite[Lemma~4.3]{DDNL5}}, to make it easily adaptable to our context:

\begin{lemma}[{\cite[Lemma~4.3]{DDNL5}}]\label{lma:exislower}
Let $u,v\in \PSH(X,\theta)$, $[u]\preceq [v]$, $\int_X\theta_{u}^n>0$. For any 
\[
b\in \bigg(1, \left(\frac{\int_X\theta_v^n}{\int_X\theta_v^n-\int_X\theta_u^n}\right)^{1/n} \bigg)
\]
there exists $h\in \PSH(X,\theta)$, such that $h+(b-1)v\leq bu$. This allows to introduce:
\begin{equation}\label{P_strange_def}
P(bu+(1-b)v):=\usc\sup\left\{h\in \PSH(X,\theta): h+(b-1)v\leq bu \right\} \in \textup{PSH}(X,\theta).
\end{equation}
\end{lemma}

To clarify, when $\int_X\theta_v^n=\int_X\theta_u^n$ the condition on $b$ in the above result is $b \in (1,\infty)$. In addition,  by \eqref{P_strange_def}, we have that $P(bu+(1-b)v) + (b-1)v \leq u$ a.e. on $X$, hence this inequality holds globally, since both the left and right-hand side are quasi-psh functions.

Next, we prove continuity results for the envelopes defined above.

\begin{proposition}\label{prop: conv_of_K_env} 
    Let $K \subseteq X$ be a compact and non-pluripolar subset.
Let $v\in C^0(K)$.
Let $u_j,u \in \PSH(X,\theta)$ such that $d_\mathcal S([u_j],[u]) \to 0$ and $\int_X \theta_u^n >0$. Then the following hold:
\begin{itemize}
    \item[\textup{(i)}] if $u_j \searrow u$ then $P_{K}[u_j]_{\mathcal{I}}(v) \searrow P_{K}[u]_{\mathcal{I}}(v)$ and $P_K[u_j](v) \searrow P_K[u](v)$.
    \item[\textup{(ii)}] if $u_j \nearrow u$ then $P_{K}[u_j]_{\mathcal{I}}(v) \nearrow P_{K}[u]_{\mathcal{I}}(v)$ a.e. and $P_K[u_j](v) \nearrow P_K[u](v)$ a.e..
\end{itemize}
\end{proposition}

The argument is very similar to that of \cite[Lemma~2.21]{DX22}.
\begin{proof}We first prove (i). Since $\int_X \theta_{u_j}^n \searrow \int_X \theta^n_u>0$ \cite[Proposition~4.8]{DDNL5}, by Lemma \ref{lma:exislower}, there exists $\alpha_j \searrow 0$ and $h_j := P(\frac{1}{\alpha_j}u + (1 - \frac{1}{\alpha_j}) u_j) \in \PSH(X,\theta)$ satisfying $(1-\alpha_j)u_j + \alpha_j h_j \leq u$. By Proposition~\ref{prop:concmono},
\[
(1-\alpha_j)P_{K}[u_j]_{\mathcal{I}}(v) + \alpha_j P_{K}[h_j]_{\mathcal{I}}(v) \leq P_{K}[(1-\alpha_j)u_j + \alpha_j h_j]_{\mathcal{I}}(v) \leq P_{K}[u]_{\mathcal{I}}(v)\,.
\]
Since $\{u_j\}_j$ is decreasing, so is $\{P_K[u_j]_{\mathcal{I}}(v)\}_j$, hence $w:= \lim_j P_K[u_j]_{\mathcal{I}}(v) \geq P[u]_{\mathcal{I}}(v)$ exists. Since $\alpha_j \to 0$ and $\sup_X P_K[h_j]_{\mathcal{I}}(v)$ is bounded, we can let $j \to \infty$ in the above estimate to conclude that $w = P_K[u]_{\mathcal{I}}(v)$. The same ideas yield that $P_K[u_j](v) \searrow P_K[u](v)$.

Proving (ii) is similar. Since $\int_X \omega_{u_j}^n \nearrow \int_X \omega^n_u>0$ \cite[Theorem~2.3]{DDNL2}, by \cite[Lemma~4.3]{DDNL5} there exists $\alpha_j \searrow 0$ and $h_j := P(\frac{1}{\alpha_j}u_j + (1 - \frac{1}{\alpha_j}) u) \in \PSH(X,\theta)$ satisfying $(1-\alpha_j)u + \alpha_j h_j \leq u_j$. By Proposition~\ref{prop:concmono},
\[
(1-\alpha_j)P_{K}[u]_{\mathcal{I}}(v) + \alpha_j P_{K}[h_j]_{\mathcal{I}}(v)  \leq  P_{K}[(1-\alpha_j) u + \alpha_j h_j]_{\mathcal{I}}(v) \leq P_{K}[u_j]_{\mathcal{I}}(v)\,.
\]
Since $\{u_j\}_j$ is increasing, so is $\{P_{K}[u_j]_{\mathcal{I}}(v)\}_j$, hence $w:= \usc\lim_j P_{K}[u_j]_{\mathcal{I}}(v) \leq P_{K}[u]_{\mathcal{I}}(v)$ exists. Since $\alpha_j \to 0$ and $\sup_X P_{K}[h_j]_{\mathcal{I}}(v)$ is bounded, we can let $j \to \infty$ in the above estimate to conclude that $w = P_{K}[u]_{\mathcal{I}}(v)$. The same proof yields that $P_K[u_j](v) \nearrow P_K[u](v)$ a.e..
\end{proof}

\subsection{An approximation result of Demailly}
Let $X$ be a compact Kähler manifold of dimension $n$. Let $\theta$ be a smooth representative of a pseudoeffective $(1,1)$-class on $X$. Let $\omega$ be a Kähler form on $X$.

Following the terminology of \cite[Defintion 2.3]{Cao14}, we recall the existence of quasi-equisingular approximation for potentials in $\PSH(X,\theta)$. As elaborated below, this result is implicit in the proof of \cite[Theorem~2.2.1]{DPS01},\cite[Theorem 3.2]{DP04} and \cite[Theorem~1.6]{Dem15}.

\begin{theorem}\label{thm:Demailly} Let $u \in \PSH(X,\theta)$.  Then there
exists $u_k^D \in \PSH(X,\theta + \varepsilon_k \omega)$ with $\varepsilon_k \searrow 0$ such that
\begin{itemize}
\item[\textup{(i)}] $u^D_k \searrow u$; 
\item[\textup{(ii)}] $[u_k^D]\in \mathcal{A}(X,\theta+\varepsilon_k\omega)$;%$\nu(T, x) - \frac{n}{2^k} \leq  \nu (T_k, x) \leq \nu(T, x)$ for all $x \in X$;
\item[\textup{(iii)}] $\mathcal I( \frac{s{2^k}}{2^k-s} u_k^D) \subseteq \mathcal I(s u) \subseteq \mathcal I(s u_k^D)$ for all $s >0$.
\end{itemize}
\end{theorem}
\begin{proof}
Parts (i) and (ii) follow from \cite[Theorem~1.6]{Dem15}. The second inclusion of (iii) follows from $u \leq u^D_k$, whereas the first inclusion of (iii) follows from \cite[Corollary~1.12]{Dem15}.
\end{proof}

As pointed out in \cite[p. 135, formula (13.14)]{Dem12} (or \cite[Theorem 3.2(iv)]{DP04}), for each $u^D_k$ in the above theorem, there exists a holomorphic  modification $\pi_k: Y_k \to X$, a smooth closed $(1,1)$-form $\beta_k$, and a $\mathbb Q$-divisor $D_k$ with snc singularities on $Y$, such that 
\begin{equation}\label{eq: neat_sing}
\theta_{u^D_k} = [D_k] + \beta_k.
\end{equation}
In particular, $u^D_k \circ \pi_k$ has neat analytic singularity type (recall \eqref{eq: analytic_def}).

In case the pseudoeffective class is induced by a line bundle, we have a related approximation result:

\begin{remark}\label{rem: Bergman_approx} In case $(L,h) \to X$ is a Hermitian line bundle with $c_1(L,h)= \{\theta\}$, $(T,h_T) \to X$ is an arbitrary Hermitian line bundle, and $\theta_u$ is a Kähler current with $[u] \in \mathcal A(X,\theta)$, it is possible to work with the following alternative approximating sequence:
\begin{equation}\label{eq: Bergman_seq_def}
\tilde u^D_k = \frac{1}{k} \log \sup_{\substack{s \in H^0(X,L^k \otimes T)\\ \int_X h^k \otimes h_T(s,s)e^{-ku} \omega^n \leq 1}} h^k  \otimes h_T(s,s)\,.
\end{equation}
For $k$ big enough, this sequence will satisfy $\tilde u^D_k + \frac{C\log k}{k}\geq u$, by the Ohsawa--Takegoshi theorem. However, it is not monotone in general. On the other hand, a stronger form of condition (iii) will hold in this case: $[u] \preceq [\tilde u_k^D] \preceq [\alpha_k u]$ for some $\alpha_k \nearrow 1$. 
\end{remark}

\begin{proof} This is a known consequence of the Briancon--Skoda theorem \cite{Dem12}, but as a courtesy to the reader we give a brief argument for the estimate $[\tilde u_k^D] \leq [\alpha_k u]$, the only part that needs to be proved. As we point out now, this actually follows from the arguments of \cite[Remark~5.9]{Dem12}. 

Let $\mathcal J$ be the coherent sheaf of holomorphic functions $g$ satisfying $|g| \leq D e^{\frac{u}{2c}}$ with $c \in \mathbb Q^+$, as in \eqref{eq: analytic_def} and $D >0$, some positive constant. As pointed out in \cite[Remark~5.9]{Dem12}, we may assume that the $f_j$'s in \eqref{eq: analytic_def} are local generators of $\mathcal J$.

Let $\pi: Y \to X$ be a smooth modification such that $\pi^{-1} \mathcal J \cdot \mathcal{O}_Y = \mathcal O(-D)$, where $D = \sum_j \lambda _j D_j$ is a normal crossing divisor on $Y$. The existence of such $\pi$ follows from Hironaka desingularization. 

Now suppose  that $s \in H^0(X,L^k \otimes T)$ satisfies $\int_X h^k \otimes h_T(s,s)e^{-ku} \omega^n \leq 1$. By pulling back we obtain
\[
    \int_Y h^k \otimes h_T(s\circ \pi,s \circ \pi)e^{-ku\circ \pi} (\pi^* \omega)^n \leq 1\,.
\]
As $u \circ \pi  \simeq c \sum_j \lambda_j \log g_j$ for some local generators $g_j$ of $\mathcal O(-D_j)$, Fubini's theorem gives that $h^k \otimes h_T(s\circ \pi,s \circ \pi)$ vanishes to order at least $\lfloor k c \lambda_j \rfloor + d$ along $D_j$, where $d$ is an absolute constant, only dependent on $\pi.$ In particular, one can find $\alpha_k \nearrow 1$ such that $h^k \otimes h_T(s\circ \pi,s \circ \pi)$ vanishes to order at least $c \alpha_k   k \lambda_j$ along $D_j$. Since  $u \circ \pi  \simeq c \sum_j \lambda_j \log g_j$, we obtain that $[\frac{1}{k} \log h^k \otimes h_T(s \circ \pi, s \circ \pi)] \preceq [\alpha_k u \circ \pi]$, which in turn gives $[\tilde u^D_k \circ \pi ] \preceq [\alpha_k u \circ \pi]$, since $H^0(X,L^k \otimes T \otimes \mathcal I(ku))$ is finite dimensional. By pushing forward, we obtain that $[\tilde u^D_k ] \preceq [\alpha_k u]$, as desired.
\end{proof}

\section{The closure of analytic singularity types in \texorpdfstring{$\mathcal S(X,\theta)$}{SX0}}\label{sec:closanasing}

In this section we only assume that $\theta$ be a smooth representative of a big $(1,1)$-cohomology class on $X$. Our goal is to prove that the $d_\mathcal S$-closure of $\mathcal A(X,\theta)$ is the space of $\mathcal I$-model singularity types, in the presence of positive mass. 
We start with an elementary lemma:
\begin{lemma} \label{lem:pullback_PI} Let $\pi: X' \to X$  be a smooth modification and $u \in \PSH(X,\theta)$. Then we have 
\[
\pi^*P^\theta[u]_{\mathcal{I}}= P^{\pi^*\theta}[\pi^*u]_{\mathcal{I}}\,.
\]
\end{lemma}
\begin{proof}
Recall that 
\begin{equation}\label{eq:inproofPu}
P^{\theta}[u]_{\mathcal{I}}=\sup\left\{\,v\in \PSH(X,\theta): v\leq 0\,,[v]\preceq_{\mathcal{I}} [u] \,\right\}\,.
\end{equation}
Let $v\in \PSH(X,\theta)$ be a candidate of the $\sup$ in \eqref{eq:inproofPu}. Then by Proposition~\ref{prop:Ilelong}, for any smooth modification $p:Y\rightarrow X$ and any $y\in Y$, $\nu(p^*v,y)\geq \nu(p^*u,y)$. In particular, for any smooth modification $q:Z\rightarrow X'$ and any $z\in Z$, we have $\nu(q^*\pi^*v,z)\geq \nu(q^*\pi^*u,z)$. By Proposition~\ref{prop:Ilelong} again, $[\pi^*v]\preceq_{\mathcal{I}} [\pi^*u]$. In particular, $\pi^*v\leq P^{\pi^*\theta}[\pi^*u]_{\mathcal{I}}$. We arrive at the inequality
\[
\pi^* (P^{\theta}[u]_{\mathcal{I}})\leq P^{\pi^*\theta}[\pi^*u]_{\mathcal{I}}\,.
\]

It remains to prove the reverse inequality.
There is a unique $h\in \PSH(X,\theta)$, such that $\pi^*h=P^{\theta}[\pi^*u]_{\mathcal{I}}$. We need to prove that $P^{\theta}[u]_{\mathcal{I}}\geq h$. 
It suffices to prove the following claim: for any $k>0$, $\mathcal{I}(ku)\supseteq \mathcal{I}(kh)$. But we already know that  $\mathcal{I}(k\pi^*u)= \mathcal{I}(k\pi^*h)$, while by \cite[Proposition~5.8]{Dem12}, 
\[
\mathcal{I}(ku)=\pi_*\left(K_{X'/X}\otimes \mathcal{I}(k\pi^*u)\right)\,,\quad \mathcal{I}(kh)=\pi_*\left(K_{X'/X}\otimes \mathcal{I}(k\pi^*h)\right) \,.
\]
Hence, we conclude that $\mathcal{I}(ku)=\mathcal{I}(kh)$.
\end{proof}

\begin{lemma}\label{lem:algebraic_PI} If $u \in \PSH(X,\theta)$ satisfies $[u] \in \mathcal A(X,\theta)$, then $[u] = [P[u]] = [P[u]_{\mathcal{I}}]$.
\end{lemma}
\begin{proof}
Since $u \sim_{\mathcal I} P[u]_{\mathcal{I}}$, we get that $[u] = [P[u]_{\mathcal{I}}]$ from \cite[Theorem~4.3]{Kim15}. Since $[u] \preceq [P[u]] \preceq [P[u]_{\mathcal{I}}]$, $[u] = [P[u]]$ also follows.
\end{proof}

\begin{proposition}\label{prop: PIlimit} 
Let $u \in \PSH(X,\theta)$. Then $P^{\theta + \varepsilon_j \omega}[u^D_j]_{\mathcal{I}} \searrow P^{\theta}[u]_{\mathcal{I}}$ as $j \to \infty$, where $u^D_j \in \PSH(X,\theta + \varepsilon_j \omega)$ is the approximating sequence of Theorem~\ref{thm:Demailly}.
Moreover, if $\theta_u$ is a Kähler current, then $P^{\theta}[u^D_j]_{\mathcal{I}} \searrow P^{\theta}[u]_{\mathcal{I}}$ as $j \to \infty$.
\end{proposition}

\begin{proof} 
We can suppose that $u \leq 0$. Since $[u^D_j] \succeq [u]$ we have that $P^{\theta + \varepsilon_j \omega}[u^D_j]_{\mathcal{I}} \geq P^{\theta + \varepsilon_j \omega}[u]_{\mathcal{I}} \geq  P^{\theta}[u]_{\mathcal{I}}$. 
Since $\{u^D_j\}_j$ is decreasing, we have that $v:= \lim_j P^{\theta + \varepsilon_j \omega}[u^D_j]_{\mathcal{I}} \in \PSH(X,\theta)$ exists and $u \leq v$.

Observe that $P^{\theta}[v]_{\mathcal{I}} =v$, since any candidate $h \in \PSH(X,\theta)$ for $P^{\theta}[v]_{\mathcal{I}}$ is also a candidate for each $P^{\theta + \varepsilon_j \omega}[u^D_j]_{\mathcal{I}}$. Hence, to finish the argument, it is enough to show that $\mathcal I(tu) = \mathcal I(tv)$ for all $t>0$. By Theorem~\ref{thm:Demailly}, for any $\delta >1$  and $t>0$ there exists $k_0(\delta,t)>0$ such that for all $k \geq k_0$ we have $\mathcal I (t \delta v) \subseteq \mathcal I (t \delta u^D_k)\subseteq \mathcal I(tu)$. Letting $\delta \searrow 1$, the Guan--Zhou strong openness theorem \cite{GuZh15} implies that $\mathcal I (t  v) \subseteq \mathcal{I}(tu)$. Since the reverse inclusion is trivial, the proof of the first assertion is finished.

To prove the second assertion, assume that $\theta_u$ is a Kähler current. Hence, for $j$ large enough, $u_j^D\in \PSH(X,\theta)$. On the other hand, observe that $P^{\theta + \varepsilon_j \omega}[u^D_j]_{\mathcal{I}}\geq P^{\theta}[u^D_j]_{\mathcal{I}}\geq P^{\theta}[u]_{\mathcal{I}}$, hence $P^{\theta}[u^D_j]_{\mathcal{I}}\searrow P^{\theta}[u]_{\mathcal{I}}$ as $j\to\infty$.
\end{proof}

We note the following important corollary of this result, that will be used numerous times in this work:

\begin{corollary}\label{cor: measureconv} Let $u \in \PSH(X,\theta)$. Then 
\[
\int_X (\theta + \varepsilon_j \omega)^n_{u^D_j} = \int_X (\theta + \varepsilon_j \omega)^n_{P^{\theta + \varepsilon_j \omega}[u^D_j]_{\mathcal{I}}} \searrow \int_X \theta_{P^{\theta}[u]_{\mathcal{I}}}^n \ \textup{ as } \ j \to \infty\,,
\]
where $u^D_j \in \PSH(X,\theta + \varepsilon_j \omega)$ is the approximating sequence of Theorem~\ref{thm:Demailly}.
\end{corollary}
\begin{proof} The equality follows from Lemma~\ref{lem:algebraic_PI} and \cite[Theorem~1.1]{WN19}.

Since $P^{\theta + \varepsilon_j \omega}[u^D_j]_{\mathcal{I}}\geq  P^{\theta}[u]_{\mathcal{I}}$, we can start with the following inequality:
\[
\varliminf_{j\to\infty} \int_X (\theta + \varepsilon_j \omega)^n_{P^{\theta + \varepsilon_j \omega}[u^D_j]_{\mathcal{I}}} \geq  \varliminf_{j\to\infty} \int_X (\theta + \varepsilon_j \omega)^n_{P^{\theta}[u]_{\mathcal{I}}}  = \int_X \theta_{P^{\theta}[u]_{\mathcal{I}}}^n\,.
\]
To finish the proof, we will argue that $\varlimsup_j \int_{X} (\theta + \varepsilon_j \omega)^n_{P^{\theta+\varepsilon_j\omega}[u_j^D]_{\mathcal{I}}} \leq \int_{X} \theta^n_{P^{\theta}[u]_{\mathcal{I}}}$. Indeed, fixing $j_0 \in \mathbb N$, we have
\begin{flalign*}
\varlimsup_{j \to \infty} \int_X (\theta + \varepsilon_j \omega)^n_{P^{\theta + \varepsilon_j \omega}[u_j^D]_{\mathcal{I}}}&=\varlimsup_{j \to \infty} \int_{\left\{P^{\theta + \varepsilon_j \omega}[u_j^D]_{\mathcal{I}} =0\right\}} (\theta + \varepsilon_j \omega)^n_{P^{\theta + \varepsilon_j \omega}[u_j^D]_{\mathcal{I}}} \\
&\leq \varlimsup_{j \to \infty} \int_{\left\{P^{\theta + \varepsilon_j \omega}[u_j^D]_{\mathcal{I}} =0\right\}} (\theta + \varepsilon_{j_0} \omega)^n_{P^{\theta + \varepsilon_{j} \omega}[u_j^D]_{\mathcal{I}}}\\
&\leq \int_{\left\{P^{\theta}[u]_{\mathcal{I}} =0\right\}} (\theta + \varepsilon_{j_0} \omega)^n_{P^{\theta}[u]_{\mathcal{I}}}\,,
\end{flalign*}
where in the first line we have used that $P^{\theta + \varepsilon_j \omega}[u_j^D]_{\mathcal{I}}=P^{\theta + \varepsilon_j \omega}[u_j^D]$ (\cite[Proposition~2.20]{DX22}) and \cite[Theorem~3.8]{DDNL2}, and in the last line we have used Proposition~\ref{prop: PIlimit} and \cite[Proposition~4.6]{DDNL5}. Letting $j_0 \to \infty$, we arrive at the desired conclusion: 
\[
\varlimsup_{j \to \infty} \int_{X} (\theta + \varepsilon_j \omega)^n_{u_j^D} \leq \varliminf_{j_0 \to \infty}\int_{\{P^{\theta}[u]_{\mathcal{I}} =0\}} (\theta + \varepsilon_{j_0} \omega)^n_{P^{\theta}[u]_{\mathcal{I}}}  =  \int_{\{P^{\theta}[u]_{\mathcal{I}} =0\}} \theta^n_{P^{\theta}[u]_{\mathcal{I}}} \leq \int_{X} \theta^n_{P^{\theta}[u]_{\mathcal{I}}}\,.
\]
\end{proof}

\begin{corollary}\label{cor:decincPI}
Let $\psi\in \PSH(X,\theta)$. The following hold: 
\begin{itemize}
    \item[\textup{(i)}] $\int_X \left(\theta+\varepsilon \omega +\ddc P^{\theta + \varepsilon \omega}[\psi]_{\mathcal{I}}\right)^n \searrow \int_X \theta_{P^{\theta}[\psi]_{\mathcal{I}}}^n$  as  $\varepsilon \searrow 0$;
    \item[\textup{(ii)}]If $\theta_\psi$ is a Kähler current, then $\int_X \left(\theta-\varepsilon \omega +\ddc P^{\theta -\varepsilon \omega}[\psi]_{\mathcal{I}}\right)^n \nearrow \int_X \theta_{P^{\theta}[\psi]_{\mathcal{I}}}^n$ as $\varepsilon\searrow 0$. 
\end{itemize}
\end{corollary}

\begin{proof} We approximate $\psi$ with $\psi^D_j \in \PSH(X,\theta + \varepsilon_j \omega)$ from Theorem~\ref{thm:Demailly}. For $\varepsilon>0$, applying Corollary~\ref{cor: measureconv} for $\psi \in \PSH(X,\theta + \varepsilon \omega)$ (but for the same approximating sequence $\psi^D_j \in \PSH(X,\theta + (\varepsilon+\varepsilon_j) \omega)$ independent of $\varepsilon$) we get that
\[
\begin{aligned}
\int_X \left(\theta+\varepsilon \omega +\ddc P^{\theta + \varepsilon \omega}[\psi]_{\mathcal{I}}\right)^n =& \lim_{j\to\infty} \int_X\left(\theta+(\varepsilon+\varepsilon_j) \omega +\ddc \psi_j^D\right)^n\,,\\
\int_X \left(\theta +\ddc P^{\theta}[\psi]_{\mathcal{I}}\right)^n =& \lim_{j\to\infty} \int_X\left(\theta + \varepsilon_j \omega +\ddc \psi_j^D\right)^n\,.
\end{aligned}
\]
Using the multi-linearity of non-pluripolar products, (i) follows. The proof of (ii) follows the same pattern and is left to the reader.
\end{proof}

\begin{proposition}\label{prop:posvol_dom_kahler_cur} 
Let $u \in \PSH(X,\theta)$ such that $\int_X \theta_u^n>0$. Then there exists $v \in \PSH(X,\theta)$ such that $u \geq v$ and $\theta_v \geq \delta \omega$ for some $\delta >0$.
\end{proposition}

\begin{proof} 
We may assume that $u\leq 0$. Since $u \leq V_\theta$ and $\int_X \theta_{V_\theta}^n \geq \int_X \theta_{u}^n>0$, by Lemma~\ref{lma:exislower}, there exists $b>0$ such that $h := P((1+b) u - b V_\theta) \in \PSH(X,\theta)$ and
\[
\frac{b}{b+1} V_\theta + \frac{1}{b+1} h \leq u\,.
\]
By \cite{Bo02}, there exists $w \in \PSH(X,\theta)$ such that $w\leq 0$ and $\theta_w \geq \delta' \omega$ for some $\delta' >0$. Since $w \leq V_\theta$, we obtain that 
\[
v:= \frac{b}{b+1} w + \frac{1}{b+1} h \leq u
\]
and $\theta_v \geq \frac{b \delta'}{b+1} \omega$.
\end{proof}

Next, we extend \cite[Theorem~2.24]{DX22} to big cohomology classes.

\begin{lemma}\label{lma:dSuku}
Let $u\in \PSH(X,\theta)$. Assume that $\theta_u$ is a Kähler current. Let $u_k^D$ be the approximation sequence in Theorem~\ref{thm:Demailly}. Then 
\begin{equation}\label{eq:dSuku}
d_{\mathcal{S}}([u^D_k],P^{\theta}[u]_{\mathcal{I}})\to 0 \text{  as  }k\to\infty\,.
\end{equation}
In particular, 
\begin{equation}\label{eq:convmass}
\lim_{k\to\infty}\int_X\theta_{u_k^D}^n= \int_X \theta_{P^{\theta}[u]_{\mathcal{I}}}^n\,.
\end{equation}
\end{lemma}
\begin{proof}
First observe that $u^D_k\in \PSH(X,\theta)$ when $k$ is large enough, so \eqref{eq:dSuku} indeed makes sense. The second assertion follows from the first and \cite[Lemma~3.7]{DDNL5}, so it suffices to prove the first. By Proposition~\ref{prop: PIlimit}, $P^{\theta}[u_k^D]_{\mathcal{I}}$ decreases to $P^{\theta}[u]_{\mathcal{I}}$ as $k\to\infty$.

Since the potentials $P^{\theta}[u_{k}^D]_{\mathcal{I}}$ are model \cite[Proposition~2.18(i)]{DX22}, by \cite[Lemma~3.6, Proposition~4.8]{DDNL5} we obtain that $d_\mathcal S([P^{\theta}[u]_{\mathcal{I}}],[P^{\theta}[u_{k}^D]]_{\mathcal{I}}) \to 0$ as $k \to \infty$. We conclude \eqref{eq:dSuku} by Lemma~\ref{lem:algebraic_PI}.
\end{proof}

\begin{theorem}\label{thm:dS_A_closure} Let $u \in \PSH(X,\theta)$ such that $\int_X \theta_u^n>0$. Then $[u] \in \overline{\mathcal A(X,\theta)}^{d_\mathcal S}$
if and only if $[P[u]] = [P[u]_{\mathcal{I}}]$. Additionally, if $[P[u]] = [P[u]_{\mathcal{I}}]$ and $\theta_u$ is a Kähler current, then the regularization sequence $\{[u^D_k]\}_k$ of Theorem~\ref{thm:Demailly} $d_\mathcal S$-converges to $[u]$.
\end{theorem}
Here the notation $\overline{\mathcal A(X,\theta)}^{d_\mathcal S}$ means the closure of $\mathcal A(X,\theta)$ in $\mathcal{S}(X,\theta)$ with respect to the $d_\mathcal S$-metric.
\begin{proof} To begin, let $v \in \PSH(X,\theta)$ be such that $v \leq u$ and $\theta_v \geq \delta \omega$ for some $\delta >0$. Such $v$ exists by Proposition~\ref{prop:posvol_dom_kahler_cur}. Let $v_t : = (1-t) v + t u, \ t \in [0,1]$. Then $\theta_{v_t}$ is a Kähler current for $t \in [0,1)$ and $v_t \nearrow u$ a.e. as $t \nearrow 1$. 

Assume first that $[P[u]_{\mathcal{I}}] = [P[u]]$. By replacing $u$ with $P [u]_{\mathcal{I}}$, we can additionally assume that $u = P[u]_{\mathcal{I}}$. By \cite[Lemma~2.21(iii)]{DX22} we obtain that $P[v_t]_{\mathcal{I}} \nearrow P[u]_{\mathcal{I}} = u$ a.e. as $ t \to 1$. In particular, by \cite[Lemma~4.1]{DDNL5} we obtain that $d_\mathcal S (P[v_t]_{\mathcal{I}},[u]) \to 0$ as $t \to 1$.

Let us fix $t \in [0,1)$. By the above, it is enough to argue that $[P[v_t]_{\mathcal{I}}] \in \overline{\mathcal A}^{d_\mathcal S}$. For this we apply the regularization method of Theorem~\ref{thm:Demailly} to $v_t$, obtaining $v_{t,k}^D \in \PSH(X,\theta)$ such that $[v_{t,k}^D] \in \mathcal A(X,\theta)$ (we used here that $\theta_{v_t}$ is a Kähler current). By Lemma~\ref{lma:dSuku}, $d_{\mathcal{S}}([v_{t,k}^D],[v_t])\to 0$ as $k\to\infty$. So $[P[v_t]_{\mathcal{I}}] \in \overline{\mathcal A}^{d_\mathcal S}$, and we conclude.

In the reverse direction, suppose there exists $[v_j] \in \mathcal A(X,\theta)$ such that $d_\mathcal S([v_j],[u]) \to 0$. By Lemma~\ref{lem:algebraic_PI}, we can assume that $v_j = P[v_j]_{\mathcal{I}}=P[v_j]$. In addition, we can assume that $u = P[u]$, since $d_\mathcal S(u,P[u])=0$ \cite[Theorem 3.3]{DDNL5}. Since $\int_X \theta_u^n >0$, after possibly restricting to a subsequence of $v_j$, we can use \cite[Theorem~5.6]{DDNL5} to conclude existence of an increasing sequence of model potentials $\{w_j\} \in \PSH(X,\theta)$ such that $w_j \leq v_j$ and   $d_\mathcal S([w_j],[u]) \to 0$. As pointed out after the statement of \cite[Theorem~5.6]{DDNL5}, after possibly taking a subsequence of the $v_j$, we can take 
\[
w_j := \lim_{k \to \infty} P(v_j, v_{j+1}, \ldots, v_{j+k})\,.
\]
Since all the $v_j$ are $\mathcal I$-model, an iterated application of Lemma~\ref{lem:rooftop_mod_I_mod} implies that so is $h_{j,k} := P(v_j, v_{j+1}, \ldots, v_{j+k})$. Moreover, since $w_j$ is the decreasing limit of the $h_{j,k}$, then $w_j$ is $\mathcal I$-model too (\cite[Lemma~2.21(i)]{DX22}). Lastly, since $u$ is the increasing limit of the $w_j$, then $u$ is $\mathcal I$-model as well (\cite[Lemma~2.21(iii)]{DX22}).
\end{proof}

\section{Proof of Theorem~\ref{thm:vol_formula_main}}\label{sec:proofmainthm}
Let $X$ be a connected compact Kähler manifold of dimension $n$.
For this section, let $T$ be an arbitrary holomorphic vector bundle on $X$, with rank $r$.

\subsection{The case of integral line bundles}
 Let $L$ be a pseudoeffective line bundle on $X$. Let $h$ be a smooth Hermitian metric on $L$ such that $\theta:=c_1(L,h)$.  We fix a Kähler form $\omega$ on $X$ such that $\omega-\theta$ is a Kähler form.

\begin{proposition}\label{prop:Tsuji_upper_bound} Suppose that $u \in \PSH(X,\theta)$. Then
\[
\varlimsup_{k \to \infty} \frac{1}{k^n}h^0(X,T \otimes L^k \otimes \mathcal I(ku)) \leq  \frac{r}{n!}\int_X \theta_{P[u]_{\mathcal{I}}}^n\,.
\]
\end{proposition}

\begin{proof} Since $P[P[u]_{\mathcal{I}}]_{\mathcal{I}} = P[u]_{\mathcal{I}}$ and $\mathcal I(s P[u]_{\mathcal{I}}) = \mathcal I(s u)$ for all $s>0$ \cite[Proposition~2.18(ii)]{DX22}, we can replace $u$ with $P[u]_{\mathcal{I}}$ to assume that $u$ is $\mathcal I$-model.

Next, we apply the regularization method of Theorem~\ref{thm:Demailly} to $u$, obtaining $u_{j}^D \in \PSH(X,\theta+\varepsilon_j\omega)$ such that $[u_{j}^D] \in \mathcal A(X,\theta+\varepsilon_j\omega)$ and $u_j^D \searrow u$. Let $\pi_k: Y_k \to X$ be the smooth resolution of singularities of \eqref{eq: neat_sing}.

By \cite[Proposition~5.8]{Dem12} and \cite[Théorème~2.1]{Bon98} applied to $q=0$ on $Y_k$ (see also \cite[Theorem~2.3.18]{MM07}), we obtain that 
\begin{flalign*}
\varlimsup_{k \to \infty} \frac{1}{k^n}h^0(X,T \otimes L^k \otimes \mathcal I(ku))  &\leq \varlimsup_{k \to \infty} \frac{1}{k^n}h^0(X,T \otimes L^k \otimes \mathcal I(ku_j^D))\\
&= \varlimsup_{k \to \infty} \frac{1}{k^n}h^0(Y,\pi_k^*T \otimes (\pi_k^*L)^k \otimes K_{Y/X} \otimes \mathcal I(ku_j^D \circ \pi_k))\\
&\leq \frac{r}{n!}\int_{Y_k(0)} \pi_k^*\theta^n_{u_j^D}= \frac{r}{n!}\int_{\pi_k(Y_k(0))} \theta^n_{u_j^D} \\
&\leq \frac{r}{n!}\int_{\pi_k(Y_k(0))} (\theta + \varepsilon_j \omega)^n_{u_j^D} \leq \frac{r}{n!}\int_{X} (\theta + \varepsilon_j \omega)^n_{u_j^D},
\end{flalign*}
where $Y_k(0) \subseteq Y_k$ is the set contained in the smooth locus of the (1,1)-current $\pi_k^*\theta_{u_j^D}$ where the eigenvalues of $\pi_k^*\theta_{u_j^D}$ are all positive. By Corollary~\ref{cor: measureconv}, $\lim_{j\to\infty} \int_{X} (\theta + \varepsilon_j \omega)^n_{u_j^D} = \int_{X} \theta^n_{u}$, finishing the argument.
\end{proof}

\begin{lemma}\label{lem:boundbelow_analyt} Let $u \in \PSH(X,\theta)$ such that $\theta_u$ is a Kähler current. Let $\beta \in (0,1)$. Then there exists $k_0 := k_0(u,\beta)$ such that for all $k \geq k_0$ there exists $v_{\beta,k} \in \PSH(X,\theta)$ satisfying the following: 
\begin{itemize}
    \item[\textup{(i)}]$P[u]_{\mathcal{I}} \geq (1-\beta) u^D_k + \beta v_{\beta,k}$;
    \item[\textup{(ii)}]$\int_X \theta_{v_{\beta,k}}^n >0$.
\end{itemize}
\end{lemma}

\begin{proof} 
Due to Lemma~\ref{lma:dSuku}, we have that $\int_X \theta_{u^D_k}^n \searrow \int_X \theta^n_{P[u]_{\mathcal{I}}}$. In particular, there exists $k_0>0$ such that 
\[
\frac{1}{\beta^{n}} < \frac{\int_X \theta^n_{u^D_{k}}}{\int_X\theta^n_{u^D_{k}}-\int_X\theta^n_{P[u]_{\mathcal{I}}}}\quad \text{for all } k\geq k_0\,.
\]
By Lemma \ref{lma:exislower} we obtain that $v_{k,\beta} := P(\frac{1}{\beta}P[u]_{\mathcal{I}} - \frac{1-\beta}{\beta}u^{D}_k) \in \PSH(X,\theta)$ and $P[u]_{\mathcal{I}} \geq (1-\beta) u^D_k + \beta v_{\beta,k}$.

Now we show that $v_{\beta,k}$ has positive mass. Pick $\beta'\in (0,\beta)$ such that 
\[
\frac{1}{\beta'^{n}} < \frac{\int_X \theta^n_{u^D_{k}}}{\int_X\theta^n_{u^D_{k}}-\int_X\theta^n_{P[u]_{\mathcal{I}}}}\quad \text{for all } k\geq k_0\,.
\]
Then $h:=P(\frac{1}{\beta'}P[u]_{\mathcal{I}} - \frac{1-\beta'}{\beta'}u^{D}_k) \in \PSH(X,\theta)$ is defined as well, and $v_{k,\beta} \geq \frac{\beta'}{\beta}h + \frac{\beta-\beta'}{\beta}u^D_k \in \PSH(X,\theta)$, implying that $\int_X \theta_{v_{k,\beta}}^n\geq \frac{(\beta-\beta')^n}{\beta^n} \int_X \theta_{u^D_k}^n\geq \frac{(\beta-\beta')^n}{\beta^n} \int_X \theta_{u}^n>0$, where we applied \cite[Theorem~1.1]{WN19} twice.
\end{proof}

\begin{proposition}\label{prop:Tsuji_lower_bound_Kahler} Suppose that $u \in \PSH(X,\theta)$ with $\theta_u > \delta \omega$ for some $\delta >0$. Then
\[
\varliminf_{j \to \infty} \frac{1}{j^n}h^0(X,T \otimes L^j \otimes \mathcal I(ju)) \geq  \frac{r}{n!}\int_X \theta_{P[u]_{\mathcal{I}}}^n\,.
\]
\end{proposition}
\begin{proof}
To start, we fix a number $\beta=p/q\in (0,\min(\delta,1))\cap \mathbb{Q}$. It suffices to show that there is a constant $C>0$, only dependent on $r$, $n$ and $\theta$, such that
\[
\varliminf_{j \to \infty} \frac{1}{j^n}h^0(X,T \otimes L^j \otimes \mathcal I(ju)) \geq  \frac{r}{n!}\int_X \theta_{P[u]_{\mathcal{I}}}^n-C\beta\,.
\]
Writing $j=aq+b$ for some $b=0,\ldots,q-1$, observe that
\[
h^0(X,T \otimes L^j \otimes \mathcal I(ju))\geq h^0\left(X,T\otimes L^{b-q}\otimes L^{(a+1)q}\otimes \mathcal{I}((a+1)qu)\right)\,.
\]
Absorbing $L^{b-q}$ into $T$, and noticing that $b-q$ can only take a finite number of values, we find that it suffices to prove the following
\begin{equation}\label{eq:inproofh0lim}
\varliminf_{j\to\infty}\frac{1}{j^nq^n}h^0(X,T \otimes L^{jq} \otimes \mathcal I(jqu))\geq \frac{r}{n!}\int_X \theta_{P[u]_{\mathcal{I}}}^n-C\beta\,,
\end{equation}
for arbitrary twisting bundle $T$.

By Lemma~\ref{lem:boundbelow_analyt}, there is $k_0>0$ depending on $\beta$ and $u$, such that for $k\geq k_0$, there exists a potential $v_{\beta,k} \in \PSH(X,\theta)$ of positive mass such that 
\[
P[u]_{\mathcal{I}} \geq w_{\beta,k}:=(1-\beta) u^D_k + \beta v_{\beta,k}\quad \text{for all }k\geq k_0\,.
\]
For big enough $k_0$ we also have
$\theta_{u^D_k} > \beta \omega \geq \beta \theta$ for all $k\geq k_0$. In particular, $u_k^D\in \PSH(X,(1-\beta)\theta)$.
We have $
H^0(X,T\otimes L^{jq}\otimes \mathcal{I}(jqu))\supseteq H^0(X,T\otimes L^{jq}\otimes \mathcal{I}(jqw_{\beta,k})),$ hence
\begin{equation}\label{eq: ineq_est1}
h^0(X,T\otimes L^{jq}\otimes \mathcal{I}(jqu))\geq h^0(X,T\otimes L^{jq}\otimes \mathcal{I}(jqw_{\beta,k})).
\end{equation}
For each fixed $k>0$, we can take a resolution of singularities $\pi:Y\rightarrow X$, such that $\pi^*u_k^D$ has neat analytic singularities along a normal crossing $\mathbb{Q}$-divisor, as described in \eqref{eq: neat_sing}. By \cite[Proposition~5.8]{Dem12} and the projection formula,
\begin{equation}\label{eq: ineq_est2}
h^0(X,T\otimes L^{jq}\otimes \mathcal{I}(jqw_{\beta,k}))=h^0(Y,\pi^*T\otimes K_{Y/X}\otimes (\pi^*L)^{jq}\otimes \mathcal{I}(jq\pi^*w_{\beta,k}))\,.
\end{equation}
Since $\int_Y(\pi^* \theta+\ddc\pi^*v_{\beta,k})^n =\int_X \theta_{ v_{\beta,k}}^n >0$, there exists a non-zero section 
\[
s_j \in H^0\left(Y,\pi^*L^{\beta jq} \otimes \mathcal I(\beta jq \pi^* v_{\beta,k})\right)=H^0\left(Y,\pi^*L^{jp} \otimes \mathcal I( jp \pi^* v_{\beta,k})\right)
\]
for all $j$ large enough, by Lemma~\ref{lem:posmasssection}. Hence applying Lemma~\ref{lem:injective} for $T: = \pi^*T\otimes K_{Y/X}$, $E_1=\pi^*L^{q-p}$, $E_2=\pi^*L^{p}$, $\chi_1:=q\pi^*u_k^D$, $\chi_2:=p\pi^*v_{\beta,k}$, $s_j:= s_j$, $\varepsilon:= \beta$, we find
\begin{flalign}\label{eq: ineq_est3}
h^0(Y,\pi^*T\otimes & K_{Y/X}\otimes \pi^*L^{jq}\otimes \mathcal{I}(jq\pi^*w_{k,\beta})) \nonumber \\
&=h^0(Y,\pi^*T\otimes K_{Y/X}\otimes \pi^*L^{(q-p)j}\otimes \pi^*L^{pj}\otimes \mathcal{I}((1-\beta)j q \pi^*u^D_k + j p \pi^* v_{\beta,k}))) \nonumber \\
&\geq h^0(Y,\pi^*T\otimes K_{Y/X}\otimes \pi^*L^{(q-p)j}\otimes \mathcal{I}(jq\pi^*u_k^D))
\end{flalign}
for $j$ large enough (depending on $k$).

Since $\theta_{u^{D}_k} > \beta \omega \geq \beta \theta$, we notice that $qu^D_k  \in \textup{PSH}(X,\theta(q-p))$. Hence, by \cite[Théorème~2.1, Corollaire~2.2]{Bon98} (see also \cite[Theorem~2.26]{DX22}), we can write the following estimates.
\begin{flalign}\label{eq: ineq_est4}
\varliminf_{j\to\infty} \frac{1}{j^n q^n}h^0(Y,\pi^*T\otimes K_{Y/X}& \otimes \pi^*L^{(1-\beta)qj}  \otimes \mathcal{I}(jq\pi^*u_k^D)) \nonumber\\
=&\varliminf_{j\to\infty} \frac{1}{j^n q^n}h^0\left(Y,\pi^*T\otimes K_{Y/X}\otimes \pi^*L^{(q-p)j}\otimes \mathcal{I}(jq\pi^*u_k^D)\right)\nonumber\\
= & \frac{r}{q^n n!}\int_Y ((q-p)\pi^*\theta+q\ddc \pi^*u_k^D)^n\\
=& \frac{r}{n!}\int_X ((1-\beta)\theta+\ddc u_k^D)^n \nonumber\\
\geq & \frac{r}{n!}\int_X \theta_{u_k^D}^n-C\beta\,, \nonumber
\end{flalign}
where $C>0$ depends only on $r,n,\theta$.
Putting together \eqref{eq: ineq_est1},\eqref{eq: ineq_est2}, \eqref{eq: ineq_est3} and \eqref{eq: ineq_est4}  we obtain
\[
\varliminf_{j \to \infty} \frac{1}{j^n}h^0(X,T \otimes L^j \otimes \mathcal I(ju)) \geq  \frac{r}{n!}\int_X \theta_{u_k^D}^n-C\beta\,.
\]
Letting $k\to\infty$ and applying Lemma~\ref{lma:dSuku}, we conclude \eqref{eq:inproofh0lim}.
\end{proof}

\begin{lemma}\label{lem:posmasssection} Suppose that $L \to X$ is a big line bundle, with smooth Hermitian metric $h$. Let $\theta=c_1(L,h)$.
Let $v \in \PSH(X,\theta)$ with $\int_X \theta_v^n >0$. Then for $m$ big enough there exists $s \in H^0(X,L^m \otimes \mathcal I(m v))$ non-vanishing.
\end{lemma}

\begin{proof} By Proposition~\ref{prop:posvol_dom_kahler_cur} there exists $w \in \PSH(X,\theta)$ such that $w \leq v$ and $\theta_w \geq \delta \omega$. By \cite[Theorem~13.21]{Dem12}, for $m$ big enough, there exists $s \in H^0(X,L^m \otimes \mathcal I(m w))$ non-zero. Since $w \leq v$, we get that $s \in H^0(X,L^m \otimes \mathcal I(m v))$. 
\end{proof}

\begin{lemma}\label{lem:injective} Suppose that $E_1,E_2,T$ are vector bundles over a  connected complex manifold $Y$,  with $\textup{rank } E_2 =1$, and $\chi_1,\chi_2$ are quasi-psh functions on $Y$, with $\chi_1$ having normal crossing divisorial singularity type. Suppose that there exists a non-zero section $s_j \in H^0(Y, E_2^{\otimes j} \otimes \mathcal I(j \chi_2))$,  for all $j$ big enough. Then for any $\varepsilon \in (0,1)$ the map $w \mapsto w \otimes s_j$  between the vector spaces
\[
H^0(Y,T \otimes E_1^{\otimes j} \otimes  \mathcal I(j\chi_1)) \to H^0\left(Y,T \otimes E_1^{\otimes j} \otimes E_2^{\otimes j} \otimes \mathcal I\left(j (1-\varepsilon)\chi_1 + j\chi_2\right)\right)
\]
is well-defined and injective, for all $j$ big enough.
\end{lemma}
\begin{proof} Suppose that the singularity type of $\chi_1$ is given by the effective normal crossing $\mathbb R$-divisor $\sum_j \alpha_j D_j$ with $\alpha_j>0$. By \cite[Remark~5.9]{Dem12} we have that 
\[
\mathcal I(j \chi_1) = \mathcal O_Y (-\sum_m \lfloor \alpha_m j \rfloor D_j)\,.
\]
We obtain that $w e^{- j(1-\varepsilon) \chi_1 }$ is bounded for any $w \in H^0(Y,T \otimes E_1^j \otimes \mathcal I(j\chi_1))$ and $j$ big enough. Since $s_j \in H^0(Y, E_2^{\otimes j} \otimes \mathcal I(j\chi_2))$, we obtain that 
\[
w \otimes s_j \in H^0\left(Y,T \otimes E_1^{\otimes j} \otimes E_2^{\otimes j} \otimes \mathcal I(j (1-\varepsilon)\chi_1 + j\chi_2)\right)\,.
\]
Injectivity of $w \mapsto w \otimes s_j$ follows from the identity theorem of holomorphic functions.
\end{proof}

\begin{theorem}\label{thm:Tsuji} Suppose that $u \in \PSH(X,\theta)$. Then
\begin{equation}\label{eq:Tsuji_eq}
\lim_{k \to \infty} \frac{1}{k^n}h^0(X,T \otimes L^k \otimes \mathcal I(ku)) = \frac{r}{n!} \int_X \theta_{P[u]_{\mathcal{I}}}^n\,.
\end{equation}
\end{theorem}

\begin{proof} Since both sides of \eqref{eq:Tsuji_eq} only depend on $P[u]_{\mathcal{I}}$, we can assume that $P[u]_{\mathcal{I}} = u$. Proposition~\ref{prop:Tsuji_upper_bound} implies \eqref{eq:Tsuji_eq} for $\int_X \theta_{u}^n=0$, so we can also assume that $\int_X \theta_{u}^n>0$. In particular, $L$ is a big line bundle and $X$ is projective.
By Proposition~\ref{prop:posvol_dom_kahler_cur}, there exists $v \leq u$ such that $\theta_v$ is a Kähler current. Let $v_t := (1-t)v + tu$. Then  $\theta_{v_t}$ is a Kähler current for $t \in [0,1)$, so we can apply Proposition~\ref{prop:Tsuji_lower_bound_Kahler} to obtain that
\[
\varliminf_{k \to \infty} \frac{1}{k^n}h^0(X,T \otimes L^k \otimes \mathcal I(ku)) \geq \varliminf_{k \to \infty} \frac{1}{k^n}h^0(X,T \otimes L^k \otimes \mathcal I(kv_t))
 \geq  \frac{r}{n!}\int_X \theta_{P[v_t]_{\mathcal{I}}}^n\,.
\]
 Letting $t \to 0$ and using \cite[Lemma~2.21(iii)]{DX22}, we obtain that 
\[
\varliminf_{k \to \infty} \frac{1}{k^n}h^0(X,T \otimes L^k \otimes \mathcal I(ku)) \geq  \frac{r}{n!} \int_X \theta_{P[u]_{\mathcal{I}}}^n\,.
\]
The reverse inequality, follows from Proposition~\ref{prop:Tsuji_upper_bound}. 
\end{proof}

\subsection{The case of \texorpdfstring{$\mathbb{R}$}{R}-line bundles}\label{subsec:Rlinebundle}

In this subsection we extend Theorem \ref{thm:Tsuji} to $\mathbb R$-line bundles. First we deal with the case of $\mathbb Q$-line bundles:

\begin{corollary}
Let $L$ be a pseudoeffective $\mathbb{Q}$-line bundle on $X$, represented by an effective $\mathbb{Q}$-divisor $D$.
Let $\theta$ be a smooth form representative of $c_1(L)$.
Let $u\in \PSH(X,\theta)$. Then
\begin{equation}\label{eq:h0rational}
\lim_{k\to\infty} \frac{1}{k^n}h^0(X,T\otimes \mathcal{O}_X(\floor*{kD})\otimes \mathcal{I}(ku))=\frac{r}{n!}\int_X \theta_{P[u]_{\mathcal{I}}}^n\,.
\end{equation}
\end{corollary}
\begin{proof}
We may assume that $D':=aD$ is a line bundle $L'$ for some $a\in \mathbb{N}$. For each $k\in \mathbb{N}$, write
\[
k=k_0a+k'\,,\quad k_0\in \mathbb{N}\,,k'\in [0,a-1)\,.
\]
Note that the difference $\floor*{kD}-k_0D'$ can represent only a finite number of different line bundles.
Hence, in order to prove \eqref{eq:h0rational}, it suffices to establish the following: for each fixed $k'\in [0,a-1)$
\[
\lim_{k_0\to\infty} \frac{1}{k_0^na^n}h^0(X,T\otimes L'^{k_0}\otimes \mathcal{I}(k_0 a u+k'u))=\frac{r}{n!}\int_X \theta_{P[u]_{\mathcal{I}}}^n\,.
\]
Observe that $\mathcal{I}(k_0 a u+k'u)\subseteq \mathcal{I}(k_0au)$, so by Theorem \ref{thm:Tsuji} we have 
\[
\varlimsup_{k_0\to\infty} \frac{1}{k_0^na^n}h^0(X,T\otimes L'^{k_0}\otimes \mathcal{I}(k_0 a u+k'u))\leq \frac{r}{n!}\int_X \theta_{P[u]_{\mathcal{I}}}^n\,.
\]
On the other hand, as $\mathcal{I}(k_0 a u+k'u)\supseteq \mathcal{I}((k_0+1)au)$, we get
\begin{flalign*}
\varliminf_{k_0\to\infty} \frac{1}{k_0^na^n}h^0(X,T & \otimes L'^{k_0}\otimes \mathcal{I}(k_0 au+k'u))\geq \varliminf \frac{1}{k_0^na^n}h^0\left(X,T\otimes L'^{k_0}\otimes \mathcal{I}((k_0 + 1)a u )\right)\\
&= \varliminf_{k_0\to\infty} \frac{1}{\big((k_0 + 1)a\big)^n}h^0\left(X,T\otimes L'^* \otimes L'^{k_0 + 1}\otimes \mathcal{I}((k_0 + 1)a u )\right)\\
&= \frac{r}{n!}\int_X \theta_{P[u]_{\mathcal{I}}}^n\,,
\end{flalign*}
where in the last step we again used Theorem \ref{thm:Tsuji}, finishing the proof.
\end{proof}
\begin{corollary}
Assume that $X$ is projective.
Let $D$ be a big $\mathbb{R}$-divisor on $X$. Let $\theta$ be a smooth form representing the cohomology class $[D]$. Let $u\in \PSH(X,\theta)$.
Then
\begin{equation}\label{eq:corbigRdiv}
\lim_{k\to\infty} \frac{1}{k^n}h^0\left(X,T\otimes \mathcal{O}_X(\floor*{kD})\otimes \mathcal{I}(ku)\right)=\frac{r}{n!}\int_X \theta_{P[u]_{\mathcal{I}}}^n\,.
\end{equation}
\end{corollary}
\begin{proof}
We first deal with the $\leq$ direction in \eqref{eq:corbigRdiv}.
Fix $\delta>0$.
Fix $\varepsilon>0$, so that 
\[
\int_X \left(\theta+\varepsilon \omega +\ddc P^{\theta + \varepsilon \omega}[u]_{\mathcal{I}}\right)^n< \int_X \theta_{P[u]_{\mathcal{I}}}^n+\delta\,.
\]
This is possible by Corollary~\ref{cor:decincPI}(i).
Take a $\mathbb{Q}$-divisor $D^{\delta}$, such that  the cohomology class $\{D^{\delta}-D\}$ has a smooth positive representative $\theta^{\delta}\leq \varepsilon\omega$. As a result,  $D^{\delta}-D$ is ample. This is possible as $X$ is projective. We have $u\in \PSH(X,\theta+\theta^{\delta})$.
%Using Lemma~ \ref{lem:injective} we can initiate the following estimates
Then $\mathcal{O}_X(\floor*{kD^{\delta}}-\floor*{kD})$ has a non-zero global section $s$ for $k$ big enough. As a result, $H^0(X,T\otimes \mathcal{O}_X(\floor*{kD})\otimes \mathcal{I}(ku)) \ni s' \mapsto s' \otimes s \in H^0(X,T\otimes \mathcal{O}_X(\floor*{kD^{\delta}})\otimes \mathcal{I}(ku))$ is injective, allowing to write the following estimates
\[
\begin{aligned}
\varlimsup_{k\to\infty} \frac{1}{k^n}h^0(X,T\otimes \mathcal{O}_X(\floor*{kD})\otimes \mathcal{I}(ku))\leq & \varlimsup_{k\to\infty} \frac{1}{k^n}h^0(X,T\otimes \mathcal{O}_X(\floor*{kD^{\delta}})\otimes \mathcal{I}(ku))\\
= & \frac{r}{n!}\int_X (\theta+\theta^\delta+\ddc P^{\theta +\theta^\delta}[u]_{\mathcal{I}})^n\\
\leq & \frac{r}{n!}\int_X (\theta+\varepsilon\omega+\ddc P^{\theta + \varepsilon \omega}[u]_{\mathcal{I}})^n\\
\leq &  \frac{r}{n!}\int_X \theta_{P[u]_{\mathcal{I}}}^n+\frac{r\delta}{n!}\,,
\end{aligned}
\]
where in the second line we have used Theorem~\ref{thm:Tsuji}.
Letting $\delta\to 0+$, we conclude the $\leq$ direction in \eqref{eq:corbigRdiv}.

For the reverse direction, we can replace $u$ by $P[u]_{\mathcal{I}}$, as in the proof of Theorem \ref{thm:Tsuji}. Hence, we can assume that $u$ is $\mathcal{I}$-model. If $\int_X \theta_u^n=0$, we are done by the previous arguments, so we can assume that $\int_X \theta_u^n>0$. 

We first treat the case where $\theta_u>\varepsilon_0 \omega$ for some $\varepsilon_0>0$. Fix $\delta>0$. 
Fix $\varepsilon\in (0,\varepsilon_0)$, so that
\[
\int_X \left(\theta-\varepsilon \omega +\ddc P^{\theta - \varepsilon \omega}[u]_{\mathcal{I}}\right)^n> \int_X \theta_{P[u]_{\mathcal{I}}}^n-\delta\,.
\]
This is possible by Corollary~\ref{cor:decincPI}(ii).
Take a $\mathbb{Q}$-divisor $D^{\delta}$ so that $\{D-D^{\delta}\}$ has a smooth  positive representative $\theta^{\delta}\leq \varepsilon\omega$. As a result, $D-D^{\delta}$ is ample. Then we have $u\in \PSH(X,\theta-\theta^{\delta})$. As before, we have the following estimates
\[
\begin{aligned}
\varliminf_{k\to\infty} \frac{1}{k^n}h^0(X,T\otimes \mathcal{O}_X(\floor*{kD})\otimes \mathcal{I}(ku))\geq & \varliminf_{k\to\infty} \frac{1}{k^n}h^0(X,T\otimes \mathcal{O}_X(\floor*{kD^{\delta}})\otimes \mathcal{I}(ku))\\
= & \frac{r}{n!}\int_X (\theta-\theta^\delta+\ddc P^{\theta -\theta^\delta}[u]_{\mathcal{I}})^n\\
\geq & \frac{r}{n!}\int_X (\theta-\varepsilon\omega+\ddc P^{\theta - \varepsilon \omega}[u]_{\mathcal{I}})^n\\
\geq &  \frac{r}{n!}\int_X \theta_{P[u]_{\mathcal{I}}}^n-\frac{r\delta}{n!}\,,
\end{aligned}
\]
where in the second line we have used Theorem~\ref{thm:Tsuji}. Letting $\delta\to 0+$, we conclude \eqref{eq:corbigRdiv} in this case.

Finally, we treat the general case. By Proposition~\ref{prop:posvol_dom_kahler_cur}, there exists $v\in \PSH(X,\theta)$, such that $v\leq u$ and $\theta_v$ is a Kähler current. Set $u_{t}:=(1-t)u+tv$ for $t\in [0,1]$. For $t\in (0,1]$, $\theta_{u_t}$ is still a Kähler current. By the special case treated above, we get
\begin{equation}\label{eq:liminfgeqPimass}
\begin{split}
\varliminf_{k\to\infty}\frac{1}{k^n}h^0\left(X,T\otimes \mathcal{O}_X(\floor*{kD})\otimes \mathcal{I}(ku)\right)\geq \lim_{k\to\infty} \frac{1}{k^n}h^0\left(X,T\otimes \mathcal{O}_X(\floor*{kD})\otimes \mathcal{I}(ku_t)\right)\\=\frac{r}{n!}\int_X \theta_{P[u_t]_{\mathcal{I}}}^n
\end{split}
\end{equation}
for $t\in (0,1]$. As $t\searrow 0$, $u_t \nearrow u$, hence $P[u_t]_{\mathcal{I}}\nearrow P[u]_{\mathcal{I}}$ a.e. by Proposition~\ref{prop: conv_of_K_env}(ii). By \cite[Theorem~2.3]{DDNL2}, $\int_{X}\theta_{P[u_t]_{\mathcal{I}}}^n\nearrow \int_{X}\theta_{P[u]_{\mathcal{I}}}^n$. Letting $t \searrow 0$ in  \eqref{eq:liminfgeqPimass}, we find the desired inequality
\[
\varliminf_{k\to\infty}\frac{1}{k^n}h^0\left(X,T\otimes \mathcal{O}_X(\floor*{kD})\otimes \mathcal{I}(ku)\right)\geq \frac{r}{n!}\int_X \theta_{P[u]_{\mathcal{I}}}^n\,.
\]
\end{proof}

\section{Envelopes of singularity types with respect to compact sets}\label{sec:env}
Let $X$ be a connected compact Kähler manifold of dimension $n$. 
For this whole section, let $K \subseteq X$ be a closed non-pluripolar set.
Let $\theta$ be a closed real $(1,1)$-form on $X$ representing a pseudoeffective cohomology class.
Let $u \in \PSH(X,\theta)$.

Let $v:K\rightarrow [-\infty,\infty)$ be a function. We introduce the following $K$-relative envelopes and their regularizations, refining the definitions from Section~\ref{subsec:singtype}:
\[
\begin{aligned}
E^{\theta}_{K}[u]_{\mathcal{I}}(v):=& \sup \left\{ h \in \PSH(X,\theta): h|_K \leq v  \textup{ and } [h]\preceq_{\mathcal{I}} [u]\right\}\,,\\
E^{\theta}_K [u](v):=& \sup \left\{h \in \PSH(X,\theta): h|_K \leq v  \textup{ and } [h] \preceq [u]\right\}\,,\\
P^{\theta}_{K}[u]_{\mathcal{I}}(v):=& \usc \big( E^{\theta}_{K}[u]_{\mathcal{I}}(v)\big)\,,  \quad P^{\theta}_K [u](v):= \usc\big(E^{\theta}_K [u](v)\big)\,.
\end{aligned}
\]
We omit $\theta$ and $K$ from the notations when there is no risk of confusion. When $v$ is bounded, neither of the above candidate sets are empty: one can always take $h=u-C$ for a large enough constant $C$.

We note the following maximum principles, that follow from the above definitions:
\begin{lemma}\label{lem: P_K_max_princ}Let $v\in C^0(K)$.
Let $h \in \PSH(X,\theta)$. Assume that $[h] \preceq [u]$, then 
\begin{equation}\label{eq: P_K_max_princ}
\sup_K (h-v) = \sup_{X\setminus \{h=-\infty\}} (h - E_K[u](v))=\sup_{X\setminus \{E_K[u](v)=-\infty\}} (h - E_K[u](v))\,.
\end{equation}
\end{lemma}
\begin{proof}
We prove the first equality at first. We write $S=\{h=-\infty\}$.

By definition, $E_K[u](v)|_K\leq v$, so 
\[
\left.\left(h-E_K[u](v)\right)\right|_{K\setminus S}\geq h|_{K\setminus S}-v|_{K\setminus S}\,.
\]
This implies that  $\sup_K (h-v) \leq \sup_{X\setminus S} (h - E_K[u](v))$.

Conversely, observe that $\sup_K(h-v)>-\infty$ as $K$ is non-pluripolar.
Let $h':=h-\sup_K (h-v)$, then $h'$ is a candidate in the definition of $E_K[u](v)$, hence $h'\leq E_K[u](v)$, namely
\[
h-\sup_K (h-v)\leq E_K[u](v)\,,
\]
the latter implies that $\sup_K (h-v) \geq \sup_{X\setminus S} (h - E_K[u](v))$, finishing the proof of the first identity.

We have $\{E_K[u](v) =-\infty\} \subseteq S$, and we notice that points in $S \setminus \{E_K[u](v) =-\infty\}$ do not contribute to the supremum of $h - E_K[u](v)$ on $X \setminus \{E_K[u](v) =-\infty\}$, hence the last equality of \eqref{eq: P_K_max_princ} also follows.
\end{proof}

Next, we make the following observations about the singularity types of our envelopes:
\begin{lemma}\label{lem: same_sing_type} For any $v \in C^0(K)$ we have $[P_K[u](v)] = [P[u]]$ and $[P_{K} [u]_{\mathcal{I}}(v)] = [P [u]_{\mathcal{I}}]$. In particular, if $[u] \in \mathcal A(X,\theta)$ then $P_K[u](v) = P_{K}[u]_{\mathcal{I}}(v)$.
\end{lemma}
\begin{proof}Let $C>0$ such that $-C \leq v \leq C$. 
Then $P[u] -C \leq P_K[u](v)$. Since $K$ is non-pluripolar, for $h \in \PSH(X,\theta)$ the condition $h|_K \leq v\leq C$ implies that $h \leq \tilde C$ on $X$ for some $\tilde C:=\tilde C(C,K)>0$ \cite[Corollary 4.3]{GZ07}. This implies that $P_K[u](v) \leq P[u] + \tilde C$, giving $[P_K[u](v)] = [P[u]]$. The exact same argument applies in case of the $P[\cdot]_\mathcal I$ envelope as well. Finally, in case $[u] \in \mathcal A(X,\theta)$, we have that $[u]=[P[u]_\mathcal I] = [P[u]]$ (Lemma~\ref{lem:algebraic_PI}). We claim that for $h\in \PSH(X,\theta)$, $[h]\preceq [u]$ if and only if $[h]\preceq_{\mathcal{I}}[u]$. This claim
immediately gives $P_K[u](v) = P_{K}[u]_{\mathcal{I}}(v)$. The forward direction of the claim is trivial, so suppose $[h]\preceq_{\mathcal{I}}[u]$. We then have $P[h]_{\mathcal{I}}\leq P[u]_{\mathcal{I}}$. This implies that $[h]\preceq [P[h]_{\mathcal{I}}]\preceq [P[u]_{\mathcal{I}}]=[u] $.
\end{proof}

%This simple lemma and \cite[Proposition~2.18(ii)]{DX22} imply the following projectivity property of  $P[\cdot]_\mathcal I(v)$:

\begin{corollary}\label{cor:projectivity} Let $u \in \textup{PSH}(X,\theta),  v \in C^0(X)$.  Then $P_K[u]_\mathcal I(v) = P_K[P_K[u]_\mathcal I(v)]_\mathcal I(v)$. 
\end{corollary}
\begin{proof}
By definition, the right-hand side is the usc regularization of
\[
\sup\left\{h\in \PSH(X,\theta): h|_K\leq v, [h]\preceq_{\mathcal{I}} P_K[u]_\mathcal I(v) \right\}.
\]
By Lemma~\ref{lem: same_sing_type} and \cite[Proposition~2.18(ii)]{DX22}, this expression can be rewritten as
\[
\sup\left\{h\in \PSH(X,\theta): h|_K\leq v, [h]\preceq_{\mathcal{I}} [u] \right\}.
\]
The usc regularization of the latter expression is just $P_K[u]_\mathcal I(v)$.
\end{proof}

\begin{lemma}\label{lma:PKoutsidepps}
Let $u\in \PSH(X,\theta)$ be a potential with positive mass. Let $v\in C^0(K)$. Let $S\subseteq X$ be a pluripolar set. Let $h\in \PSH(X,\theta)$, $[h]\preceq [u]$. Assume that $h$ has positive mass and $h|_{K\setminus S}\leq v|_{K\setminus S}$, then $h\leq P_K[u](v)$.
\end{lemma}
\begin{proof}
By the global Josefson theorem (\cite[Theorem~7.2]{GZ05}), there is $\chi\in \PSH(X,\theta)$, such that $S\subseteq \{\chi=-\infty\}$.
We claim that we can choose $\chi$ so that $\chi\leq h$. In fact, since $\int_X \theta_h^n>0$, fixing some $\chi$ and $\varepsilon>0$ small enough, we have
\[
\int_X \theta^n_{\varepsilon \chi+(1-\varepsilon)V_\theta}+\int_X \theta_{h}^n>\int_X \theta_{V_\theta}^n\,.
\]
Thus, by \cite[Lemma~5.1]{DDNL5}, we have $P(\varepsilon \chi+(1-\varepsilon)V_\theta,h)\in \PSH(X,\theta)$. Since $P(\varepsilon \chi+(1-\varepsilon)V_\theta,h) \leq \varepsilon \chi$, the claim is proved by replacing $\chi$ with $P(\varepsilon \chi+(1-\varepsilon)V_\theta,h)$.

Fix $\chi\leq h$ as above. For any $\delta\in (0,1)$, we have
\[
(1-\delta)h|_K+\delta \chi|_K\leq v\,,\quad [(1-\delta)h+\delta \chi]\preceq [u]\,.
\]
Hence, $(1-\delta)h+\delta \chi\leq P_K[u](v)\,.$
Letting $\delta \searrow 0$, we conclude that $h\leq P_K[u](v)$.
\end{proof}

\begin{corollary}\label{cor:PKtoPX}
Let $u\in \PSH(X,\theta)$ be a potential with positive mass. Let $v\in C^0(K)$. Then
\[
P_K[u](v)=P_X[u]\left( P_K[V_\theta](v) \right)\,.
\]
\end{corollary}
\begin{proof}
It is clear that $P_K[u](v) \leq P_X[u]\left( P_K[V_\theta](v) \right)$. For the reverse direction, it suffices to prove that any $h\in \PSH(X,\theta)$ such that $[h]\preceq [u]$, $h\leq P_K[V_\theta](v)$ satisfies $h\leq P_K[u](v)$. As $u$ has positive mass, we can assume that $h$ has positive mass as well. Let $S=\{P_K[V_\theta](v)>E_K[V_\theta](v)\}$. By \cite[Theorem~7.1]{BT82}, $S$ is a pluripolar set. Observe that $h|_{K\setminus S}\leq v|_{K\setminus S}$, hence by Lemma~\ref{lma:PKoutsidepps}, $h\leq P_K[u](v)$ and we conclude.
\end{proof}

The next result motivates our terminology to call the measures $\theta_{P_K[u](v)}^n$ the \emph{partial equilibrium measures} of our context:

\begin{lemma}\label{lma:balayage}
Let $v\in C^0(K)$.
Let $u\in \PSH(X,\theta)$. 
Then $\theta_{P_K[u](v)}^n$ does not charge $X\setminus K$. Moreover, $P_K[u](v)|_K=v$ a.e. with respect to $\theta_{P_K[u](v)}^n$. More precisely, we have
\begin{equation}\label{eq:thetaPKuv}
\theta_{P_K[u](v)}^n\leq \mathds{1}_{K\cap \{P_K[u](v)=P_K[V_\theta](v)=v\}}\,\theta_{P_K[V_\theta](v)}^n\,.    
\end{equation}
\end{lemma}

\begin{proof}
First we address the case when $u = V_\theta$.

Let $S\subseteq X$ be a closed pluripolar set, such that $V_{\theta}$ is locally bounded on $X\setminus S$. 

For the first assertion, it suffices to show that $\theta_{P_K[V_\theta](v)}^n$ does not charge any open ball $B\Subset X\setminus (S \cup K)$. 

By Choquet's lemma, we can take an increasing sequence $h_j\in \PSH(X,\theta)$ converging to $P_K[V_{{\min}}](v)$ a.e. and  $h_j|_K \leq v$. By \cite[Proposition~9.1]{BT82}, we can find $w_j\in \PSH(X,\theta)$, such that $(\theta + \ddc w_j|_B)^n=0$ and $w_j$ agrees with $h_j$ outside $B$. Note that $w_j$ is clearly increasing and $w_j\geq h_j$, along with $w_j|_K \leq v$. It follows that $w_j$ converges to $P_K[V_\theta](v)$ as well. By continuity of the Monge--Ampère operator along increasing bounded sequences \cite[Theorem~2.3]{DDNL2}, we find that $\theta^n_{P_K[V_\theta](v)}$ does not charge $B$, as desired.

For the second assertion, let $x\in (X\setminus S) \cap K $ be a point such that $P_K[V_\theta](v)(x)<v(x)-\varepsilon$ for some $\varepsilon>0$. Let $B$ be a ball centered at $x$, small enough so that $\theta$ has a local potential on $B$, allowing us to identify $\theta$-psh functions with psh functions (on $B$).
By shrinking $B$, we can further guarantee
\begin{enumerate}
    \item $\overline B\subseteq X\setminus S$.
    \item $P_K[V_\theta](v)|_{\overline B}<v(x)-\varepsilon$.
    \item $v|_{\overline B\cap K}>v(x)-\varepsilon$.
\end{enumerate}
Construct the sequences $h_j$, $w_j$ as above. On $B$, by choosing a local potential of $\theta$, we may identify $h_j$, $w_j$ with the corresponding psh functions in a neighborhood of  $\overline B$. 
By 2. we have $w_j\leq v(x)-\varepsilon$ on $\partial B$, hence by the comparison principle, $w_j|_B\leq v(x)-\varepsilon$. By 3. we have $w_j|_{B\cap K}\leq v|_{B\cap K}$. Thus, we conclude that $\theta^n_{P_K[V_\theta](v)}$ does not charge $B$, as in the previous paragraph.

For the general case, we can assume $\int_X \theta_u^n >0$. Indeed, due to Lemma~\ref{lem: same_sing_type}, we have that $\int_X \theta_{P_K[u](v)}^n = \int_X \theta_u^n,$ hence there is nothing to prove if $\int_X \theta_u^n =0$. By Corollary~\ref{cor:PKtoPX}, $P_K[u](v) = P_X[u](P_K[V_\theta](v))$.
Now \cite[Theorem~3.8]{DDNL2} gives 
\[
\theta_{P_K[u](v)}^n \leq \mathds{1}_{\{P_K[u](v) = P_K[V_\theta](v)\}} \theta_{P_K[V_\theta](v)}^n \leq \mathds{1}_{\{P_K[u](v) = v\}} \theta_{P_K[V_\theta](v)}^n\,,
\]
where in the last inequality we have used the first part of the argument.
\end{proof}

\begin{corollary}\label{cor:suppthetan}
Let $v\in C^0(K)$.
Let $u\in \PSH(X,\theta)$. Then $\theta^n_{P_K[u](v)}$ (resp. $\theta^n_{P_K[u]_{\mathcal{I}}(v)}$) does not charge $(X\setminus K)\cup \{P_K[u](v)<v\}$ (resp. $(X\setminus K)\cup \{P_K[u]_{\mathcal{I}}(v)<v\}$).
\end{corollary}

\begin{proof}
The first part of the corollary follows from  Lemma~\ref{lma:balayage}. For the second part, we can assume that $\int_X \theta^n_{P_K[u]_{\mathcal{I}}(v)}>0$, otherwise there is nothing to prove. By definition, we have
$P_K[u]_{\mathcal{I}}(v)=P_K[P[u]_{\mathcal{I}}]_{\mathcal{I}}(v)$. 

Next we show that $P_K[P[u]_{\mathcal{I}}]_{\mathcal{I}}(v)=P_K[P[u]_{\mathcal{I}}](v)$. The inequality $P_K[P[u]_{\mathcal{I}}]_{\mathcal{I}}(v)\geq P_K[P[u]_{\mathcal{I}}](v)$ is trivial. By Lemma \ref{lem: same_sing_type} we get that $[P_K[P[u]_{\mathcal{I}}]_{\mathcal{I}}(v)]=[P[u]_{\mathcal{I}}]$. Due to Choquet's lemma, we get that $P_K[P[u]_{\mathcal{I}}]_{\mathcal{I}}(v) \leq v$ on $K \setminus S$, where $S$ is pluripolar. 
As a result, due to the non-vanishing mass assumption, Lemma \ref{lma:PKoutsidepps} allows to conclude that $P_K[P[u]_{\mathcal{I}}]_{\mathcal{I}}(v)\leq P_K[P[u]_{\mathcal{I}}](v)$.

Since $P_K[P[u]_{\mathcal{I}}]_{\mathcal{I}}(v)=P_K[u]_{\mathcal{I}}(v)$, we get that $\theta^n_{P_K[u]_{\mathcal{I}}(v)}$ does not charge  $(X\setminus K)\cup \{P_K[u]_{\mathcal{I}}(v)<v\}$, using the first part of the corollary.
\end{proof}

\begin{proposition}\label{prop:PKdependsonmodeltype}
Let $u\in \PSH(X,\theta)$ be a potential with positive mass. Let $v\in C^0(K)$. Then
\begin{equation}\label{eq: interm_eq}
    P_K[u](v)=P_K[P[u]](v)\,.
\end{equation}
In particular, $P_K[u](v) = P_K[P_K[u](v)](v)$.
\end{proposition}
\begin{proof}
It is obvious that $P_K[u](v)\leq P_K[P[u]](v)$. We to prove the reverse inequality. As $P_K[u](v)$ and $P_K[P[u]](v)$ have the same singularity types (Lemma~\ref{lem: same_sing_type}), by the domination principle \cite[Corollary~3.10]{DDNL2}, it suffices to show that $P_K[u](v)\geq P_K[P[u]](v)$ a.e. with respect to $\theta_{P_K[u](v)}^n$.
By \eqref{eq:thetaPKuv}, $P_K[u](v)=P_K[V_\theta](v)=v$ a.e. with respect to $\theta_{P_K[u](v)}^n$. Hence, $P_K[P[u]](v)=v$ a.e. with respect to $\theta_{P_K[u](v)}^n$. We conclude that $P_K[u](v)=P_K[P[u]](v)$.

Finally, that $P_K[u](v) = P_K[P_K[u](v)](v)$ follows from Lemma \ref{lem: same_sing_type} and \eqref{eq: interm_eq}.
\end{proof}

\begin{lemma} \label{lem: global_env_approx} Fix a Kähler form $\omega$ on $X$.
For $v\in C^0(K)$ there exists an increasing bounded sequence $\{v^-_j\}_j$ in $C^\infty(X)$ and a decreasing bounded sequence $\{v^+_j\}_j$ in $C^\infty(X)$, such that for all $u \in \PSH(X,\theta)$ with $\int_X \theta_u^n>0$, and $\delta\in [0,1]$ we have: \vspace{0.1cm}\\
\noindent (i) $P_X^{\theta+\delta \omega}[u](v_j^+) \searrow P_K^{\theta+\delta \omega}[u](v)$. \vspace{0.1cm}\\
\noindent (ii) $P_X^{\theta+\delta \omega}[u](v^-_j) \nearrow P_K^{\theta+\delta \omega}[u](v)$ a.e. \vspace{0.1cm}\\
\noindent (iii) $\sup_X |v_j^-|\leq C(\|v\|_{C^0(K)},K,\theta+\omega)$ and $\sup_X |v_j^+| \leq C(\|v\|_{C^0(K)},K,\theta+\omega)$.
\end{lemma}

\begin{proof} 
We fix $\delta \in [0,1]$.  First we prove the existence of $\{v^-_j\}_j$. Let 
\[
C_{K,v} := \sup\{\sup_X w:w\in \PSH(X,\theta+\omega),  w|_K \leq v\}\,.
\]
Since $K$ is non-pluripolar, we have that $C_{K,v} \in \mathbb R$.
Now let  $\tilde v: X \to \mathbb R$ so that $\tilde v|_K = v$ and $\tilde v|_{X \setminus K} = C_{K,v}+1$.
Since $\tilde v$ is lsc, there exists an increasing and uniformly bounded sequence $\{v^-_j\}_j$ in $C^\infty(X)$, such that $v^-_j \nearrow \tilde v$.

Observe that $P^{\theta + \delta \omega}_X[u](v^-_j)$ is increasing in $j$, and $P^{\theta + \delta \omega}_X[u](v^-_j)\leq P^{\theta + \delta \omega}_K[u](v)$.
To prove that $P^{\theta + \delta \omega}_X[u](v^-_j) \nearrow P^{\theta + \delta \omega}_K[u](v)$ a.e., let $w$ be a  candidate for $P^{\theta + \delta \omega}_K[u](v)$ such that $\sup_K (w-v) < 0$. Then, since $w$ is usc  and $w < \tilde v$, by Dini's lemma there exists $j_0$ such that $w < v^-_j$ for $j \geq j_0$, i.e. $w \leq P^{\theta + \delta \omega}_X[u](v^-_j)$, proving existence of $\{v^-_j\}_j$.

Next, we prove the existence of $\{v^+_j\}_j$.
Since $h : = \max(P^{\theta + \omega}_K[V_{\theta+\omega}](v),\inf_K v -1)$ is usc, there exists a decreasing and uniformly bounded sequence $\{v^+_j\}_j$ in  $C^\infty(X)$, such that $v^+_j \searrow h$. 
Trivially, $\chi := \lim_{j\to\infty} P^{\theta + \delta \omega}_X[u](v^+_j) \geq P^{\theta + \delta \omega}_K[u](v)$. In particular, $\chi$ has positive mass, since it has the same singularity types as $P^{\theta + \delta \omega}_K[u](v)$ (Lemma \ref{lem: same_sing_type}). We introduce
\[
S:=\left\{E^{\theta + \omega}_K[V_{\theta+\omega}](v)<P^{\theta + \omega}_K[V_{\theta+\omega}](v)\right\}\,.
\]
By \cite[Theorem~7.1]{BT82},  $S$ is a pluripolar set. 
Observe that $P^{\theta + \delta \omega}_X[u](v_j^+)\leq v_j^+$ for all $j$. Thus, $\chi\leq h$. On the other hand, $h\leq v$ on $K\setminus S$. So in particular, $\chi|_{K\setminus S}\leq v|_{K\setminus S}$. By Lemma \ref{lem: same_sing_type} we also have that $[\chi] = [P_K^{\theta + \delta \omega}[u](v)]$. Hence, by Lemma~\ref{lma:PKoutsidepps}, $\chi\leq P^{\theta + \delta \omega}_K[P^{\theta + \delta \omega}_K[u](v)](v)=P^{\theta + \delta \omega}_K[u](v)$, where we also used the last statement of Proposition \ref{prop:PKdependsonmodeltype}.
\end{proof}
We recall the relative Monge--Amp\`ere energy  $I_{[u]}^{\theta} : \mathcal{E}^{1}(X,\theta;P[u]) \to \mathbb R$ from \cite{DDNL2}:
\begin{equation}
I_{[u]}^{\theta}(\varphi):=\frac{1}{n+1}\sum_{i=0}^n \int_X (\varphi-P[u]) \,\theta_{\varphi}^i\wedge \theta_{P[u]}^{n-i}\,.
\end{equation}
Using integration by parts (see \cite{Xia19b}, \cite[Theorem 1.2]{Lu21}, \cite[Theorem 2.6]{Vu20}, c.f. \cite[Theorem 1.14]{BEGZ10}), the argument of \cite[Corollary~4.2]{BB10} can be reproduced line by line to yield the following cocycle property: for $\varphi_1,\varphi_2\in \mathcal{E}^{1}(X,\theta;P[u]) $ such that $[\varphi_i]=[P[u]]$ for $i=1,2$, we have
\begin{equation}\label{eq:cocy}
    I_{[u]}^{\theta}(\varphi_1)-I_{[u]}^{\theta}(\varphi_2)=\frac{1}{n+1}\sum_{i=0}^n \int_X (\varphi_1-\varphi_2) \,\theta_{\varphi_1}^i\wedge \theta_{\varphi_2}^{n-i}\,.
\end{equation}

Following \cite{BB10}, we define the \emph{partial equilibrium energy} $\mathcal{I}_{[u],K}^{\theta}$ of  $v \in C^0(K)$:
\begin{equation}
\mathcal{I}_{[u],K}^{\theta}(v):=I_{[u]}^{\theta}(P_K[u](v))\,.
\end{equation}
In \cite{BB10} the authors used the symbol $\mathcal E$ for the above quantity. Due to potential confusion with the notation for (relative) full mass classes (that also uses the symbol $\mathcal E$), we use the symbol $\mathcal I$ instead.

Next we extend \cite[Theorem~B]{BB10}, using the arguments of \cite[Proposition~4.32]{Dar19}, itself reproducing ideas from \cite{LN15}:

\begin{proposition}\label{prop: differential_P} Let $K \subseteq X$ be a closed non-pluripolar set, $v,f \in C^0(K)$ and $u \in \PSH(X,\theta)$ satisfies $\int_X \theta_u^n>0$. 
Then $t \mapsto \mathcal{I}_{[u],K}^{\theta}(v+tf)$, $t \in \mathbb R$ is differentiable and
\begin{equation}\label{eq:ddtI}
\frac{\mathrm{d}}{\mathrm{d}t}\mathcal{I}_{[u],K}^{\theta}(v+tf) = \int_K f\, \theta^n_{P_K[u](v+tf)}\,.
\end{equation}
\end{proposition}

In this work, we will only need this result in case $[u] \in \mathcal A(X,\theta)$.

\begin{proof} 
Note that it suffices to prove \eqref{eq:ddtI} at $t=0$, which is equivalent to
\begin{equation}\label{eq: to_prove_1}
\lim_{t\to 0}\frac{I_{[u]}^{\theta}(P_K[u](v+tf))-I_{[u]}^{\theta}(P_K[u](v))}{t}=\int_K f\,\theta^n_{P_K[u](v)}\,.
\end{equation}
By switching $f$ to $-f$, we may assume that $t>0$ in the above limit.

By the comparison principle \cite[Proposition~3.5]{DDNL2} and \eqref{eq:cocy}, we find
\[
\begin{aligned}
\mathcal{I}_{[u],K}^{\theta}(v+tf)-\mathcal{I}_{[u],K}^{\theta}(v)
=&\frac{1}{n+1}\sum_{i=0}^n\int_X (P_K[u](v+tf)-P_K[u](v))\,\theta_{P_K[u](v+tf)}^i\wedge\theta_{P_K[u](v)}^{n-i}\\
\leq & \int_X(P_K[u](v+tf)-P_K[u](v))\,\theta_{P_K[u](v)}^n\,.
\end{aligned}
\]
By Lemma~\ref{lma:balayage}, $\int_X(P_K[u](v+tf)-P_K[u](v))\,\theta_{P_K[u](v)}^n\leq t\int_K f\,\theta_{P_K[u](v)}^n\,. $
Thus, we get the inequality,
\[
\varlimsup_{t\to 0+}\frac{I_{[u]}^{\theta}(P_K[u](v+tf))-I_{[u]}^{\theta}(P_K[u](v))}{t}\leq\int_K f\,\theta^n_{P_K[u](v)}\,.
\]
Similarly, we have 
\begin{flalign*}
I_{[u]}^{\theta}(P_K[u](v+tf))-I_{[u]}^{\theta}(P_K[u](v))&\geq \int_X(P_K[u](v+tf)-P_K[u](v))\,\theta_{P_K[u](v+tf)}^n\\
&\geq t\int_K f\,\theta_{P_K[u](v+tf)}^n\,.
\end{flalign*}
Together with the above, this implies \eqref{eq: to_prove_1}.
\end{proof}

In the next lemma, we prove convergence results for the partial equilibrium energy:

\begin{lemma}\label{lma:equiconv}
Let $v\in C^0(K)$ and $u \in \PSH(X,\theta)$ with $\int_X \theta_u^n>0$. Let $v_j^-$, $v_j^+$ be the sequences constructed in Lemma~\ref{lem: global_env_approx}. Then
\[
\lim_{j\to\infty}\mathcal{I}_{[u],X}^{\theta}(v_j^-)=\mathcal{I}_{[u],K}^{\theta}(v)\,,\quad \lim_{j\to\infty}\mathcal{I}_{[u],X}^{\theta}(v_j^+)=\mathcal{I}_{[u],K}^{\theta}(v)\,.
\]
\end{lemma}
\begin{proof}
This follows from Lemma~\ref{lem: same_sing_type}, Lemma~\ref{lem: global_env_approx} and \cite[Theorem~2.3]{DDNL2}.
\end{proof}

\section{Quantization of partial equilibrium measures}\label{sec:quant}

In this section, we give a proof for Theorem \ref{thm:conv_Bergman_equi_main}. Throughout the section $L \to X$ is a pseudoeffective line bundle and $h$ is a Hermitian metric on $L$ such that $\theta := c_1(L,h)$. Let $T \to X$ be a Hermitian line bundle on $X$ with a smooth Hermitian metric $h_T$. We normalize the Kähler metric $\omega$ on $X$ so that
\[
\int_X \omega^n = 1\,.
\]

\subsection{Bernstein--Markov measures}
Let $K\subseteq X$ be a closed non-pluripolar subset. 
Let $v$ be a measurable function on $K$ and let $\nu$ be a positive Borel probability measure on $K$.
We introduce the following functions on $H^0(X,L^k \otimes T)$, with values possibly equaling $\infty$:
\[
\begin{aligned}
N^k_{v,\nu}(s) :=& \left(\int_K h^k \otimes h_T(s,s) e^{-kv} \,\mathrm{d}\nu\right)^{\frac{1}{2}}\,,\\
N^k_{v,K}(s) :=& \sup_{K\setminus \{v=-\infty\}} \big( h^k \otimes h_T(s,s)e^{-kv}\big)^{\frac{1}{2}}\,.
\end{aligned}
\]

We start with the following elementary observation:
\begin{lemma}\label{lma:mononorm}
Let $v_1\leq v_2$ be two measurable functions on $X$. Assume that $\{v_1=-\infty\}=\{v_2=-\infty\}$.
Then for any $s\in H^0(X,L^k\otimes T)$ and any $k>0$, we have
\[
N_{v_1,K}^k(s)\geq N_{v_2,K}^k(s)\,.
\]
\end{lemma}

If $\nu$ puts no mass on $\{v=-\infty\}$ then we always have
\begin{equation}\label{eq:Nkinfcomp}
N^k_{v,\nu}(s)\leq N^k_{v,K}(s)\,.
\end{equation}

We recall terminology introduced in \cite{BB10}, providing a natural context in which the converse of \eqref{eq:Nkinfcomp} holds, with subexpontential growth.
A \emph{weighted subset} of $X$ is a pair $(K,v)$ consisting of a closed non-pluripolar subset $K\subseteq X$ and a function $v\in C^0(K)$.

Let $(K,v)$ be a weighted subset of $X$.
A positive Borel probability measure $\nu$ on $K$ is \emph{Bernstein--Markov} with respect to $(K,v)$ if  for each $\varepsilon>0$, there is a constant $C_\varepsilon>0$ such that
\begin{equation}\label{eq:BM}
     N^k_{v,K}(s) \leq C_\varepsilon e^{\varepsilon k} N^k_{v,\nu}(s)
\end{equation}
for any $s \in H^0(X,L^k \otimes T)$ and any $k\in \mathbb{N}$.
We write $\BM(K,v)$ for the set of Bernstein--Markov measures with respect to $(K,v)$. As pointed out in \cite{BBWN11}, any volume form measure on $X$ is Bernstein--Markov with respect to $(X,v)$, with $v \in C^\infty(X)$.

\begin{proposition}\label{prop:BMNimplynorm}
Assume that $(K,v)$ is a weighted subset of $X$, then
\begin{itemize}
    \item[\textup{(i)}]  $N^k_{v,K}$ is a norm on $H^0(X,L^k\otimes T)$. 
    \item[\textup{(i)}] For any $\nu\in \BM(K,v)$, $N^k_{v,\nu}$ is a norm on $H^0(X,L^k\otimes T)$. 
\end{itemize}
\end{proposition}
\begin{proof}
(i)
As $v$ is bounded, $N^k_{v,K}$ is clearly finite on $H^0(X,L^k\otimes T)$. In order to show that it is a norm, it suffices to show that for any $s\in H^0(X,L^k\otimes T)$,  $N^k_{v,K}(s)=0$ implies that $s=0$. In fact, we have $s|_K=0$, hence $s=0$ by the connectedness of $X$.

(ii)
As $v$ is bounded, clearly $N^k_{v,\nu}$ is finite and satisfies the triangle inequality. Non-degeneracy follows from the fact that $N^k_{v,K}$ is a norm and \eqref{eq:BM}.
\end{proof}

\subsection{Partial Bergman kernels}
In this section, with the terminology and context of the previous section, we fix a weighted subset $(K,v)$ of $X$ and $\nu\in \BM(K,v)$. We introduce the associated partial Bergman kernels: for any $k\in \mathbb{N}$, $x\in K$,
 \begin{equation}
     B^k_{v,u, \nu}(x) := \sup \left\{h^k \otimes h_T(s,s)e^{-kv}(x): N^k_{v,\nu}(s,s) \leq 1\,,  s \in H^0(X,L^k \otimes T \otimes \mathcal I(ku)) \right\}\,.
 \end{equation}
The associated partial Bergman measures on $X$ are identically zero on $X \setminus K$, and on $K$ are defined in the following manner 
\begin{equation}
\beta^k_{v,u,\nu} : = \frac{n!}{k^n} B^k_{v,u,\nu} \,\mathrm{d}\nu\,.
\end{equation}
Observe that
\begin{equation}\label{eq:intbeta}
    \int_K \beta^k_{v,u,\nu}=\frac{n!}{k^n}h^0(X,L^k \otimes T \otimes \mathcal I(ku))\,.
\end{equation}

Our aim is to show the following weak convergence result:
\begin{equation}\label{eq: want_to_prove}
\beta^k_{v,u, \nu} \rightharpoonup \theta_{P_{K}[u]_{\mathcal{I}}(v)}^n\,, \quad \text{as } k \to \infty\,.
\end{equation}

We focus momentarily on the case when $\mathrm{d}\nu = \omega^n$ and $K = X$. That \eqref{eq: want_to_prove} holds in this particular case follows from \cite[Theorem 1.4]{RWN17}. Relying on the recent paper \cite{DNT19}, we give here a short proof of this result, borrowing ideas from \cite{Brm11} as well:
\begin{proposition}\label{prop: smooth_weak_conv} Let $u \in \PSH(X,\theta)$ such that $\theta_u$ is a Kähler current and $[u] \in \mathcal A(X,\theta)$. If $v \in C^\infty(X)$, then $\beta^k_{v,u,\omega^n} \rightharpoonup \theta_{P_{X}[u]_{\mathcal{I}}(v)}^n=\theta_{P_X[u](v)}^n$ as $k\to\infty$.
\end{proposition}

\begin{proof} That $\theta_{P_{X}[u]_{\mathcal{I}}(v)}^n=\theta_{P_X[u](v)}^n$ follows from Lemma \ref{lem: same_sing_type}.
We start with noticing that as $k\to\infty$,
\[
\beta^k_{v,u,\omega^n} \leq \beta^k_{v,V_\theta,\omega^n} \rightharpoonup \theta_{P_X[V_\theta](v)}^n=\mathds{1}_{\{v = P_X[V_\theta](v)\}}\theta_v^n\,, 
\]
where the convergence follows from \cite[Theorem~1.2]{Brm11}, and the last identity is due to \cite[Corollary~3.4]{DNT19}. Let $\mu$ be the weak limit of a subsequence of $\beta^k_{v,u,\omega^n}$, then we obtain that 
\begin{equation}\label{eq: Bergmanmeasure}
\mu \leq \lim_{k\to\infty} \beta^k_{v,V_\theta,\omega^n} = \mathds{1}_{\{v = P_X[V_\theta](v)\}}\theta_v^n\,.
\end{equation}
Let $s \in H^0(X,L^k \otimes T \otimes \mathcal I(ku))$ be a section such that $N^k_{v,\omega^n}(s,s) \leq 1$. Then by \cite[Lemma~4.1]{Brm11}, there exists $C>0$ such that $h^k\otimes h_T(s,s)e^{-kv} \leq B^k_{v,u,\omega^n} \leq B^k_{v,V_\theta,\omega^n} \leq k^n C$. This implies that
\[
\frac{1}{k}\log h^k \otimes h_T(s,s) \leq v + \frac{\log C }{k} + n \frac{\log k}{ k }\,.
\]
However, we also have that $[\frac{1}{k}\log h^k \otimes h_T(s,s)] \leq [\tilde u_k^D] \leq \alpha_k [u]$, where $\tilde u^D_k$ is as defined in Remark~\ref{rem: Bergman_approx}, and $\alpha_k \in (0,1)$ is also from the notation of Remark~\ref{rem: Bergman_approx}. Let $\varepsilon >0$. We notice that $\frac{1}{k}\log h^k \otimes h_T(s,s) \in \PSH(X, \theta + \varepsilon \omega)$ for all $k \geq k_0(\varepsilon)$.
In particular,
\[
\frac{1}{k}\log h^k \otimes h_T(s,s)-\frac{\log C}{k}-n\frac{\log k}{k}\leq P_{X}^{\theta+\varepsilon\omega}[\alpha_k u](v)\,.
\]
Now taking supremum over all candidates $s$, we obtain that 
\begin{equation}\label{eq: smooth_Berg_est}
B^k_{v,u,\omega^n} \leq C k^n e^{k (P_X^{\theta +\varepsilon \omega}[\alpha_k u](v)-v)}\,,\quad k \geq k_0\,. 
\end{equation}

We claim that $\mu$ does not put mass on $\{P_X^{\theta +\varepsilon \omega}[u](v) < v\}$ for any $\varepsilon >0$. Since $P_X^{\theta +\varepsilon \omega}[\alpha_k u](v) \searrow P_X^{\theta +\varepsilon \omega}[u](v)$ (Proposition~\ref{prop: conv_of_K_env}), we get that  $\{P_X^{\theta +\varepsilon \omega}[\alpha_k u](v) < v\} \nearrow \{P_X^{\theta +\varepsilon \omega}[u](v) < v\}$. 
As a result, to argue the claim, it suffices to show that $\mu$ does not put mass on the set $\{P_X^{\theta +\varepsilon \omega}[\alpha_k u](v) < v\}$ for any $k$. 
Note that the latter set is open, hence  \eqref{eq: smooth_Berg_est} implies our claim. 

Since $u \in \mathcal A(X,\theta)$, we have that $P_X^{\theta +\varepsilon \omega}[ u](v) = [u]$ for all $\varepsilon\geq0$ by Lemma \ref{lem: same_sing_type}. As a result, $P_X^{\theta +\varepsilon \omega}[ u](v) \searrow P_X^{\theta }[u](v)$, and we can let $\varepsilon \searrow 0$ to conclude that $\mu$ does not put mass on $\{P_X[u](v) < v\} = \bigcup_{\varepsilon>0} \{P_X^{\theta +\varepsilon \omega}[u](v) < v\}$. Putting this together with \eqref{eq: Bergmanmeasure}, we obtain that 
\[
\mu \leq  \mathds{1}_{\{P_X[u](v) = v\}}\theta_{v}^n = \theta_{P_X[u](v)}^n\,, 
\]
where the last equality is due to \cite[Corollary~3.4]{DNT19}. Comparing total masses (via \eqref{eq:intbeta} and Theorem~\ref{thm:vol_formula_main}), we conclude that $\mu=\theta_{P_X[u](v)}^n$. As $\mu$ is an arbitrary limit point of $\beta^k_{v,u,\omega^n}$, we conclude that $\beta^k_{v,u,\omega^n}$ converges weakly to $\theta_{P_X[u](v)}^n$, as $k\to\infty$.
\end{proof}

Next, let $\Nm(H^0(X,L^k \otimes T \otimes \mathcal I(ku)))$ be the space of $\mathbb C$-norms on the vector space $H^0(X,L^k \otimes T \otimes \mathcal I(ku))$, and let $\mathcal L_{k,u} : \Nm(H^0(X,L^k \otimes T \otimes \mathcal I(ku))) \to \mathbb R$ be the \emph{partial Donaldson functional}, extending the definition from \cite{BB10}:
\[
\mathcal L_{k,u}(H) = \frac{n!}{k^{n+1}} \log \frac{\vol\{s \ : \ H(s) \leq 1\}}{\vol \{s \ : \ N^k_{0,\omega^n}(s) \leq 1\}}\,,
\]
where $\vol$ is simply the Euclidean volume.

\begin{proposition}\label{prop: quant_I_smooth}
Let $w,w' \in C^0(X)$, $u \in \PSH(X,\theta)$ such that $\theta_u$ is a Kähler current and $[u] \in \mathcal A(X,\theta)$. Then
\begin{equation}\label{eq:LdiffonXsmoothmeasure}
\lim_{k\to\infty} \left(\mathcal L_{k,u}(N^k_{w,\omega^n}) - \mathcal L_{k,u}(N^k_{w',\omega^n}) \right)= \mathcal{I}^{\theta}_{[u],X}(w)-\mathcal{I}^{\theta}_{[u],X}(w')\,.
\end{equation}
In particular,
\begin{equation}\label{eq:LdiffonXsmoothmeasure2}
\lim_{k\to\infty}\mathcal L_{k,u}(N^k_{w,\omega^n})=\mathcal{I}^{\theta}_{[u],X}(w)\,.
\end{equation}
\end{proposition}

\begin{proof}
First observe that by Proposition~\ref{prop:BMNimplynorm}, for any $k>0$, $N^k_{w,\omega^n}$ and $N^k_{w',\omega^n}$ are both norms, hence the  expressions inside the limit in \eqref{eq:LdiffonXsmoothmeasure} make sense.

To start, we make the following classical observation:
\begin{flalign*}
\mathcal L_{k,u}(N^k_{w,\omega^n}) - \mathcal L_{k,u}(N^k_{w',\omega^n})&= \int_0^1 \frac{\mathrm{d}}{\mathrm{d}t} \mathcal L_{k,u}(N^k_{w + t(w'-w),\omega^n})\,\mathrm{d}t\\
& = \int_0^1 \int_X (w'-w) \,\beta^k_{w + t(w'-w),u,\omega^n}\, \mathrm{d}t\,.
\end{flalign*}
By Proposition~\ref{prop: smooth_weak_conv}, we have
\[
\lim_{k\to\infty}\int_X (w'-w) \,\beta^k_{w + t(w'-w),u,\omega^n}=\int_X (w'-w)\,\theta_{P_X[u](w + t(w'-w))}^n\,.
\]
By Theorem \ref{thm:vol_formula_main} we have  $|\int_X (w'-w)\beta^k_{w + t(w'-w), u , \omega^n}| \leq C \sup_X |w - w'|$. Hence, by the dominated convergence theorem we obtain that 
\[
\begin{aligned}
\lim_{k\to\infty} \left(\mathcal L_{k,u}(N^k_{w,\omega^n}) - \mathcal L_{k,u}(N^k_{w',\omega^n})\right) &= \int_0^1 \int_X (w'-w) \,\theta_{P_X[u](w + t(w'-w))}^n \,\mathrm{d}t\\
& = \mathcal{I}^{\theta}_{[u],X}(w)-\mathcal{I}^{\theta}_{[u],X}(w')\,,
\end{aligned}
\]
where in the last line we have used Proposition~\ref{prop: differential_P}.

Finally, \eqref{eq:LdiffonXsmoothmeasure2} is just a special case of \eqref{eq:LdiffonXsmoothmeasure} with $w'=0$.
\end{proof}

\begin{lemma}\label{lem:BML}
Let $u\in \PSH(X,\theta)$.
Let $(K,v)$ be a weighted subset of $X$. Let $\nu\in \BM(K,v)$. Then
\begin{equation}\label{eq: Bern_Mark_implies}
\lim_{k\to\infty} \big( \mathcal L_{k,u}(N^k_{v,K}) - \mathcal L_{k,u}(N^k_{v,\nu})\big) = 0\,.
\end{equation}
\end{lemma}
This is a direct consequence of the definition of Bernstein--Markov measures \eqref{eq:BM}.

\begin{corollary}\label{cor:Ninfdifflim}
Let $w \in C^0(X)$, $u \in \PSH(X,\theta)$ such that $\theta_u$ is a Kähler current and $[u] \in \mathcal A(X,\theta)$. Then
\[
\lim_{k\to\infty}\mathcal L_{k,u}(N^k_{w,X})=\mathcal{I}^{\theta}_{[u],X}(w)\,.
\]
\end{corollary}
\begin{proof}
This follows from Lemma~\ref{lem:BML} and Proposition~\ref{prop: quant_I_smooth} and the fact that $\omega^n \in \BM(X,0)$.
\end{proof}

Using these preliminary facts we extend Proposition~\ref{prop: quant_I_smooth} for much less regular data, again relying on ideas from \cite{BB10}:

\begin{proposition}\label{prop: quant_I_algebraic_BM} Let $u \in \PSH(X,\theta)$ such that $\theta_u$ is a Kähler current and $[u] \in \mathcal A(X,\theta)$. Let $(K,v)$, $(K',v')$ be two weighted subsets of $X$.
Then
\begin{equation}\label{eq:LkdiffconvtoI}
\lim_{k\to\infty} \big( \mathcal{L}_{k,u}(N^k_{v,K}) - \mathcal{L}_{k,u}(N^k_{v',K'}) \big) = \mathcal{I}^{\theta}_{[u],K}(v)-\mathcal{I}^{\theta}_{[u],K'}(v')\,.
\end{equation}
In particular,
\begin{equation}\label{eq:Lkconv}
\lim_{k\to\infty}\mathcal{L}_{k,u}(N^k_{v,K})=\mathcal{I}^{\theta}_{[u],K}(v)\,.
\end{equation}
\end{proposition}

\begin{proof}
First observe that by Proposition~\ref{prop:BMNimplynorm}, for any $k>0$, $N^k_{v,K}$ and $N^k_{v',K'}$ are both norms, hence the expressions inside the limit in \eqref{eq:LkdiffconvtoI} make sense. Moreover, \eqref{eq:Lkconv} is just a special case of \eqref{eq:LkdiffconvtoI} for $K'=X$ and $v'= 0$.

To prove \eqref{eq:LkdiffconvtoI} it is enough to show that for any fixed $w \in C^\infty(X)$ we have
\begin{equation}\label{eq:inproofLdiffsmw}
\lim_{k\to\infty} \big( \mathcal L_{k,u}(N^k_{v,K}) - \mathcal L_{k,u}(N^k_{w,\omega^n}) \big) = \mathcal{I}^{\theta}_{[u],K}(v) - \mathcal{I}^{\theta}_{[u],X}(w)\,.
\end{equation}

For $\varepsilon  \in (0,1)$ small enough we have that $\theta_{(1-\varepsilon)u}$ is still a Kähler current. Let us fix such $\varepsilon$, along with an arbitrary $\varepsilon'  \in (0,1)$.

Let $\{v^-_{j}\}_j, \{v^+_{j}\}_j$  be the sequence of smooth global functions constructed in Lemma~\ref{lem: global_env_approx} for the data $(K,v)$.

By Remark~\ref{rem: Bergman_approx} there exists $k_0(\varepsilon,\varepsilon') \in \mathbb N$ such that $[\frac{1}{k} \log h^k \otimes h_T (s,s)] \preceq [(1-\varepsilon)u]$, as well as $\frac{1}{k} \log h^k \otimes h_T (s,s) \in \PSH(X,\theta + \varepsilon'\omega)$ for any $s \in H^0(X,T\otimes L^k  \otimes \mathcal I(ku)), \ k \geq k_0(\varepsilon,\varepsilon')$. 

In particular, Lemma~\ref{lem: P_K_max_princ} gives  that
\[
N^k_{E^{\theta + \varepsilon'\omega}_K[(1-\varepsilon)u](v),X}(s)=N^k_{v,K}(s)\,.
\]
\[N^k_{E^{\theta + \varepsilon'\omega}_X[(1-\varepsilon)u](v_{j}^-),X}(s)=N^k_{v_{j}^-,X}(s)\,.
\]
\[
N^k_{E^{\theta + \varepsilon'\omega}_X[(1-\varepsilon)u](v_{j}^+),X}(s)=N^k_{v^+_{j},X}(s)\,.
\]
As $E^{\theta + \varepsilon'\omega}_X[(1-\varepsilon)u](v_{j}^-) \leq E^{\theta + \varepsilon'\omega}_K[(1-\varepsilon)u](v) \leq E^{\theta + \varepsilon'\omega}_X[(1-\varepsilon)u](v_{j}^+)$, by Lemma~\ref{lma:mononorm} we have
\[
N^k_{v^+_{j},X}(s) \leq N^k_{v,K}(s) \leq N^k_{v_{j}^-,X}(s)\, , \, s \in H^0(X,T \otimes L^k  \otimes \mathcal I(ku)), \ k \geq k_0(\varepsilon,\varepsilon').
\]
Composing with $\mathcal{L}_{k,u}$ we arrive at:
\[
\mathcal{L}_{k,u}(N^k_{v^-_{j} ,X})\leq  \mathcal{L}_{k,u}(N^k_{v,K})\leq  \mathcal{L}_{k,u}(N^k_{v^+_{j} ,X})\,, \ k \geq k_0(\varepsilon,\varepsilon').
\]
For any $j>0$, by Corollary~\ref{cor:Ninfdifflim} we get
\[
\begin{aligned}
\mathcal{I}^{\theta}_{[u],X}(v^-_{j}) - \mathcal{I}^{\theta}_{[u],X}(w)=& \lim_{k\to\infty} \left(\mathcal{L}_{k,u}(N^k_{v^+_{j} ,X})-\mathcal{L}_{k,u}(N^k_{w,X})\right)\\
\leq & \varliminf_{k\to\infty}\left(\mathcal{L}_{k,u}(N^k_{v,K})- \mathcal{L}_{k,u}(N^k_{w,X})\right)\\
\leq & \varlimsup_{k\to\infty}\left(\mathcal{L}_{k,u}(N^k_{v,K})- \mathcal{L}_{k,u}(N^k_{w,X})\right)\\
\leq & \lim_{k\to\infty} \left(\mathcal{L}_{k,u}(N^k_{v^-_{j} ,X})-\mathcal{L}_{k,u}(N^k_{w,X})\right)\\
=&  \mathcal{I}^{\theta}_{[u],X}(v^+_{j}) - \mathcal{I}^{\theta}_{[u],X}(w) \,.
\end{aligned}
\]
Using Lemma~\ref{lma:equiconv}, we can let $j \to \infty$ to arrive at 
\[
\begin{split}
\mathcal{I}^{\theta}_{[u],K}(v)  -\mathcal{I}^{\theta}_{[u],K}(w) &\leq \varliminf_{k\to\infty}\left(\mathcal{L}_{k,u}(N^k_{v,K})- \mathcal{L}_{k,u}(N^k_{w,X})\right)\\
&\leq\varlimsup_{k\to\infty}\left(\mathcal{L}_{k,u}(N^k_{v,K})- \mathcal{L}_{k,u}(N^k_{w,X})\right)\\
& \leq \mathcal{I}^{\theta}_{[u],K}(v)  -\mathcal{I}^{\theta}_{[u],K}(w)\,.
\end{split}
\]
Hence, \eqref{eq:inproofLdiffsmw} follows.
\end{proof}

\begin{corollary} \label{cor:LktoI}
Let $u \in \PSH(X,\theta)$ such that $\theta_u$ is a Kähler current and $[u] \in \mathcal A(X,\theta)$. Let $(K,v)$ be a weighted subset of $X$. Assume that $\nu\in \BM(K,v)$.
Then
\[
\lim_{k\to\infty}  \mathcal{L}_{k,u}(N^k_{v,\nu})= \mathcal{I}^{\theta}_{[u],K}(v) \,.
\]
\end{corollary}
\begin{proof}
Our claim follows from Proposition~\ref{prop: quant_I_algebraic_BM} and Lemma~\ref{lem:BML}.
\end{proof}

\begin{proposition}\label{prop:weakconvana}
Suppose that $u \in \PSH(X,\theta)$ with $[u] \in \mathcal A(X,\theta)$, and we assume that $\theta_u$ is a Kähler current. Let $(K,v)$ be a weighted subset of $X$. Let $\nu\in \BM(K,v)$. Then $\beta^k_{v,u,\nu} \rightharpoonup \theta_{P_{K}[u]_{\mathcal{I}}(v)}^n=\theta_{P_K[u](v)}^n$ weakly as $k \to \infty$.
\end{proposition}

The following proof is similar to that of \cite[Theorem~B]{BBWN11}. 
\begin{proof}
For $w\in C^0(X)$, let 
\[
f_k(t)=\mathcal{L}_{k,u}(N^k_{v+tw,\nu})\,,\quad g(t):=\mathcal{I}^{\theta}_{[u],K}(v+tw)\,.
\]
By Corollary~\ref{cor:LktoI} $\varliminf_{k\to\infty} f_k(t)= g(t)$. Note that $f_k$ is concave by H\"older's inequality (see \cite[Proposition~2.4]{BBWN11}), so by \cite[Lemma~7.6]{BB10},
$\lim_{k\to\infty} f_k'(0)=g'(0)$,
which is equivalent to $\beta^k_{v,u,\nu} \rightharpoonup \theta_{P_K[u](v)}^n$, by Proposition~\ref{prop: differential_P}.
\end{proof}

Next, we deal with the case of Kähler currents:
\begin{proposition} \label{prop:mainKahcurr}
Suppose that $u \in \PSH(X,\theta)$ such that $\theta_u$ is a Kähler current. 
Let $(K,v)$ be a weighted subset of $X$ and $\nu\in \BM(K,v)$. Then $\beta^k_{v,u,\nu} \rightharpoonup \theta_{P_{K}[u]_{\mathcal{I}}(v)}^n$ as $k \to \infty.$
\end{proposition}

\begin{proof} Let $\mu$ be the weak limit of a subsequence of $\beta^k_{v,u,\nu}$. We claim that 
\begin{equation}\label{eq:inproofmuleq}
\mu\leq \theta_{P_{K}[u]_{\mathcal{I}}(v)}^n\,.
\end{equation}
Observe that this claim implies the conclusion. In fact, by Theorem~\ref{thm:vol_formula_main}, we have equality of the total masses, so equality holds in \eqref{eq:inproofmuleq}. As $\mu$ is an arbitrary subsequential limit of the weak compact sequence $\{\beta^k_{v,u,\nu}\}_k$, we get that $\beta^k_{v,u,\nu}\rightharpoonup \theta_{P_{K}[u]_{\mathcal{I}}(v)}^n$ as $k\to\infty$.

We prove the claim.
Let $\{u^D_j\}_j$ be the approximation sequence of Theorem~\ref{thm:Demailly}. By Lemma~\ref{lem: same_sing_type}, Lemma~\ref{lma:dSuku}, we know that $d_\mathcal S([u^D_j],[P_{K}[u]_{\mathcal{I}}]) = d_\mathcal S([u^D_j],[P_{K}[u]_{\mathcal{I}}(v)]) \to 0\,.$
In particular, 
\begin{equation}\label{eq:inproofeqmass}
\lim_{j\to\infty} \int_X \theta_{P_{K}[u_j^D]_{\mathcal{I}}(v)}^n= \int_X  \theta_{P_{K}[u]_{\mathcal{I}}(v)}^n\,.
\end{equation}
We know that $\theta_{u^D_j}$ are Kähler currents, for high enough $j$. Since $u \leq u^D_j$, we trivially obtain $\beta^k_{v,u,\nu} \leq \beta^k_{v,u^D_j,\nu}$ for any $k \geq 1$.
As $\nu\in \BM(K,v)$,
by Proposition~\ref{prop:weakconvana}, $\mu \leq \theta_{P_{K}[u^D_j]_{\mathcal{I}}(v)}^n$, for any $j \geq 1$ fixed. 
By Proposition~\ref{prop: conv_of_K_env}, $P_{K}[u^D_j]_{\mathcal{I}}(v) \searrow P_{K}[u]_{\mathcal{I}}(v)$ as $j\to\infty$. Hence, by \eqref{eq:inproofeqmass} and \cite[Theorem~2.3]{DDNL2}, \eqref{eq:inproofmuleq} follows.
\end{proof}

Finally, the main result:

\begin{theorem} \label{thm: pBMconvergence} Suppose that $u \in \PSH(X,\theta)$. Let $(K,v)$ be a weighed subset of $X$, let $\nu \in \BM(K,v)$. Then $\beta^k_{v,u,\nu} \rightharpoonup \theta_{P_{K}[u]_{\mathcal{I}}(v)}^n$ as $k \to \infty$.
\end{theorem}

\begin{proof} By Lemma~\ref{lem: same_sing_type} and \cite[Proposition~2.18]{DX22} we have that 
\begin{flalign*}
H^0\left(X,L^k \otimes T \otimes \mathcal I(ku)\right) &= H^0\left(X,L^k \otimes T \otimes \mathcal I(kP[u]_{\mathcal{I}})\right)\\
 &= H^0\left(X,L^k \otimes T \otimes \mathcal I(kP[u]_{\mathcal{I}}(v))\right)\,. 
\end{flalign*}
This allows us to replace $u$ with $P_{K}[u]_{\mathcal{I}}(v)$. In addition, by Theorem~\ref{thm:vol_formula_main} we can also assume that $\int_X \theta_{u}^n > 0$, otherwise there is nothing to prove.

By Proposition~\ref{prop:posvol_dom_kahler_cur}, there exists $u_j\in \PSH(X,\theta)$, such that $u_j\nearrow u$ a.e. and $\theta_{u_j}$ are Kähler currents. This gives $\beta^k_{v,u_j,\nu} \leq \beta^k_{v,u,\nu}$.
Let $\mu$ be the weak limit of a subsequence of $\beta^k_{v,u,\nu}$. Then by Proposition~\ref{prop:mainKahcurr}, $\theta_{P_{K} [u_j]_{\mathcal{I}}(v)}^n\leq \mu$.
By Proposition~\ref{prop: conv_of_K_env} and \cite[Theorem~2.3]{DDNL2} we have that $\theta_{P_{K} [u_j]_{\mathcal{I}}(v)}^n \nearrow \theta_{P_{K} [u]_{\mathcal{I}}(v)}^n$. Hence,
\begin{equation}\label{eq:inproofmulower}
\theta_{P_{K} [u]_{\mathcal{I}}(v)}^n\leq \mu\,.
\end{equation}
A comparison of total masses (\eqref{eq:intbeta} and Theorem~\ref{thm:vol_formula_main}) gives that equality holds in \eqref{eq:inproofmulower}. As $\mu$ is an arbitrary subsequential limit of the weak compact sequence $\{\beta^k_{v,u,\mu}\}_k$, we obtain that $\beta^k_{v,u,\nu} \rightharpoonup \theta_{P_{K}[u]_{\mathcal{I}}(v)}^n$ as $k\to\infty$.
\end{proof}

\begingroup
\setstretch{1.1}
\setlength\bibitemsep{0pt}
\setlength\biblabelsep{0pt}
\printbibliography
\endgroup

\small{
\noindent {\sc Department of Mathematics, University of Maryland, USA}\\
{\tt tdarvas@umd.edu}\vspace{0.1in}\\

\noindent {\sc Institut de Mathématiques de Jussieu, Sorbonne Universit\'e, France}\\
{\tt mingchen@imj-prg.fr}\vspace{0.1in}
}
\end{document}